\documentclass[a4paper,10pt]{scrartcl}
\pdfoutput=1
\KOMAoptions{parskip=half}
\KOMAoptions{twoside=semi}
\usepackage[english]{babel}
\usepackage[latin1]{inputenc} 
\usepackage{setspace} 
\usepackage[plainpages=false, pdfborder={0 0 0.1}]{hyperref} 
\usepackage{enumitem} 
\usepackage{bbm} \newcommand{\Eins}{\mathbbm{1}}
\usepackage[leqno]{amsmath}
\usepackage{amsfonts,amssymb,amstext,amsthm} 
\usepackage{mathtools} 
\usepackage{mathrsfs} 
\usepackage{fancyhdr} 
\usepackage{nicefrac}

\theoremstyle{plain}
\swapnumbers
\newtheorem{thmAMS}{Theorem}[section]
\newtheorem{lemAMS}[thmAMS]{Lemma}
\newtheorem{propAMS}[thmAMS]{Proposition}
\newtheorem{corAMS}[thmAMS]{Corollary}

\theoremstyle{definition}
\newtheorem{defnAMS}[thmAMS]{Definition}

\newtheorem{exmpAMS}[thmAMS]{Example}
\newtheorem{remAMS}[thmAMS]{Remark}
\newtheorem{clmAMS}[thmAMS]{Claim}

\newenvironment{thm}{\vspace{3pt}\begin{thmAMS}}{\end{thmAMS}}
\newenvironment{lem}{\vspace{3pt}\begin{lemAMS}}{\end{lemAMS}}
\newenvironment{prop}{\vspace{3pt}\begin{propAMS}}{\end{propAMS}}
\newenvironment{cor}{\vspace{3pt}\begin{corAMS}}{\end{corAMS}}
\newenvironment{defn}{\vspace{3pt}\begin{defnAMS}}{\end{defnAMS}}
\newenvironment{exmp}{\vspace{3pt}\begin{exmpAMS}}{\end{exmpAMS}}
\newenvironment{rem}{\vspace{3pt}\begin{remAMS}}{\end{remAMS}}

\newenvironment{prf}{\begin{proof}}{\end{proof}\vspace{3pt}}

\newcommand{\TAG}[1]{\addtocounter{thmAMS}{1}\tag{\textbf{\thethmAMS}}\label{#1}}
\newcommand{\refeqn}[1]{\hyperref[#1]{Equation (\ref*{#1})}}

\newcommand{\refX}[2]{\hyperref[#2]{#1~\ref*{#2}}}
\newcommand{\refXohneNum}[2]{\hyperref[#2]{#1}}
\newcommand{\refXX}[3]{\hyperref[#2]{#1~\ref*{#2}#3}}
\newcommand{\refNumX}[2]{\hyperref[#1]{\ref*{#1}#2}}

\newcommand{\refKlammern}[1]{\hyperref[#1]{(\ref*{#1})}}
\newcommand{\refKlammernX}[2]{\hyperref[#1]{(\ref*{#1}#2)}}

\newcommand{\refsec}[1]{\hyperref[#1]{Section~\ref*{#1}}}
\newcommand{\refchap}[1]{\hyperref[#1]{Chapter~\ref*{#1}}}

\newcommand{\refthm}[1]{\hyperref[#1]{Theorem~\ref*{#1}}}
\newcommand{\refthmX}[2]{\hyperref[#1]{Theorem~\ref*{#1}#2}}
\newcommand{\reflem}[1]{\hyperref[#1]{Lemma~\ref*{#1}}}
\newcommand{\reflemX}[2]{\hyperref[#1]{Lemma~\ref*{#1}#2}}
\newcommand{\refprop}[1]{\hyperref[#1]{Proposition~\ref*{#1}}}
\newcommand{\refpropX}[2]{\hyperref[#1]{Proposition~\ref*{#1}#2}}
\newcommand{\refcor}[1]{\hyperref[#1]{Corollary~\ref*{#1}}}
\newcommand{\refcorX}[2]{\hyperref[#1]{Corollary~\ref*{#1}#2}}
\newcommand{\refdefn}[1]{\hyperref[#1]{Definition~\ref*{#1}}}
\newcommand{\refdefnX}[2]{\hyperref[#1]{Definition~\ref*{#1}#2}}
\newcommand{\refrem}[1]{\hyperref[#1]{Remark~\ref*{#1}}}
\newcommand{\refremX}[2]{\hyperref[#1]{Remark~\ref*{#1}#2}}
\newcommand{\refexmp}[1]{\hyperref[#1]{Example~\ref*{#1}}}
\newcommand{\refexmpX}[2]{\hyperref[#1]{Example~\ref*{#1}#2}}
\newcommand{\refclm}[1]{\hyperref[#1]{Claim~\ref*{#1}}}
\newcommand{\refclmX}[2]{\hyperref[#1]{Claim~\ref*{#1}#2}}

\newcommand{\skal}[2]{\langle #1, #2 \rangle} 
\newcommand{\norm}[1]{\lVert #1 \rVert} 
\newcommand{\abs}[1]{\left\vert #1 \right\vert} 
\newcommand{\set}[1]{\left\{#1\right\}} 
\newcommand{\Set}[2]{\left\{#1\ \middle\vert\ #2 \right\}} 
\newcommand{\clSPAN}[2]{\ensuremath{\overline{\mathrm{span\vphantom{k}}}\left\{#1\ \middle\vert\ #2 \right\}}} 

\newcommand{\TO}{\longrightarrow}
\newcommand{\folgt}{\Rightarrow}
\newcommand{\eqvt}{\Leftrightarrow}
\newcommand{\Geht}{\xrightarrow} 

\newcommand{\hr}[1]{\mathscr{#1}}
\newcommand{\hrB}{\hr{B}}

\newcommand{\hrH}{\hr{H}}
\newcommand{\hrK}{\hr{K}}
\newcommand{\hrT}{\hr{T}}
\newcommand{\BH}{\hrB(\hrH)}

\newcommand{\KH}{\hrK(\hrH)}
\renewcommand{\TH}{\hrT(\hrH)} 
\newcommand{\trnorm}[1]{\lVert #1 \rVert_{\tr}} 

\newcommand{\cA}{\mathcal{A}}

\newcommand{\cF}{\mathcal{F}}
\newcommand{\cG}{\mathcal{G}}

\newcommand{\cP}{\mathcal{P}}

\newcommand{\Zst}{\mathfrak{S}}
\newcommand{\fix}[1]{\cF(#1)}

\newcommand{\linfty}{\ell^\infty}
\newcommand{\Leins}{\mathrm{L}^1}
\newcommand{\Lzwei}{\mathrm{L}^2}
\newcommand{\Linfty}{\mathrm{L}^\infty}

\newcommand{\CC}{\mathbb{C}}

\newcommand{\NN}{\mathbb{N}}

\newcommand{\RR}{\mathbb{R}}
\newcommand{\TT}{\mathbb{T}}
\newcommand{\ZZ}{\mathbb{Z}}

\newcommand{\st}{\ensuremath{\mathrm{stop}}}

\DeclareMathOperator{\tr}{tr}
\DeclareMathOperator{\id}{Id}

\DeclareMathOperator{\supp}{supp}

\DeclareMathOperator{\co}{co}
\DeclareMathOperator{\Span}{span}
\DeclareMathOperator*{\wslim}{\ensuremath{\mathrm{w*-lim}}}

\DeclareMathOperator*{\stlim}{\st-lim}

\newcommand{\eps}{\varepsilon}
\newcommand{\Fi}{\varphi}
\newcommand{\ro}{\varrho}
\newcommand{\ChFkt}{\chi}

\newcommand{\ptr}{p_{\mathrm{Tr}}}
\newcommand{\ppr}{p_{\mathrm{R}}}
\newcommand{\pnr}{p_{\mathrm{R}_0}}
\newcommand{\cAtr}{\cA_{\mathrm{Tr}}}
\newcommand{\cApr}{\cA_{\mathrm{R}}}
\newcommand{\cApot}{\cA_{\mathrm{pot}}}
\newcommand{\Tpr}{T_{\mathrm{R}}}

\begin{document}

\title{\vspace*{-4ex}A Coherent Approach to\\[-0.4ex] Recurrence and Transience for\\[-0.4ex] Quantum Markov Operators}
\author{%
Andreas G\"{a}rtner\thanks{E-mail address: gaertner@mathematik.tu-darmstadt.de}\;\quad and\quad %
Burkhard K\"{u}mmerer\thanks{E-mail address: kuemmerer@mathematik.tu-darmstadt.de}\\%
\begin{small}Fachbereich Mathematik, Technische Universit\"{a}t Darmstadt,\end{small}\\[-1.4ex]%
\begin{small}Schlo{\ss}gartenstr.~7, 64289 Darmstadt, Germany\end{small}%
}
\date{November 29, 2012}

\maketitle

\begin{abstract}\vspace*{-4ex}
	We present a coherent approach to recurrence and transience, starting from a version of the Riesz decomposition theorem for superharmonic elements.
	Our approach allows straightforward proofs of some known results, entails new theorems, and has applications to other aspects of completely positive operators: It leads to a classification of idempotent Markov operators, thereby identifying concretely the Choi-Effros product, which can be introduced on the range of these maps. We obtain an abstract Poisson integral and a representation theorem for idempotent entanglement breaking channels.
\end{abstract}

\section{Introduction}\label{sec:intro}

In the present paper we investigate the long term behavior of discrete time Quantum Markov Processes by characterizing recurrent and transient parts in terms of the corresponding transition operators.
In classical probability theory the notions of recurrence and transience provide a fundamental tool in the study of Markov processes. However, due to the lack of points in the state space of a quantum Markov process, these notions do not allow an immediate and unique generalization to the non-commutative situation, and thus their study is still in its infancy.

In the following we present a coherent approach to non-commutative versions of these notions for discrete time quantum Markov processes in terms of the corresponding transition operators, thereby incorporating several partial results scattered around the literature.
Inspired by the work of F. Fagnola, R. Rebolledo, and V. Umanit\`{a} (cf.\ \cite{fr}, \cite{fag}, \cite{uma}) and F. Haag (cf.\ \cite{flo}) we use notions of classical (probabilistic) potential theory (see, e.g., \cite{rev} or \cite{dm}) to define potentials, transient projections, and recurrent projections for quantum Markov operators;  positive recurrent projections are defined via support projections of stationary normal states of such operators (see, e.g., \cite{ehk}, \cite{fv}, \cite{groh}, \cite{luc}, \cite{fag}, or \cite{uma}).

Different notions of recurrence and transience have been studied, for example, in \cite{acc}, \cite{moh}, \cite{rz}, or \cite{gvww}.

At the starting point of our approach stands a non-commutative version of the Riesz decomposition theorem, which we put to use several times throughout this paper. After having explored the basic notions of recurrence and transience under various aspects, it turns out that our results have applications to idempotent Markov operators, to the Choi-Effros product (cf.\ \cite{ce}), and to entanglement breaking channels (cf.\ \cite{hol}, \cite{hsr}). We also draw a connection between the theory of non-commutative Poisson boundaries (cf.\ \cite{izugrp}, \cite{izu}) and weak* mean ergodicity (see, e.g., \cite{kn}).

The paper is organized as follows: Apart from notation and preliminaries, \refsec{sec:prelim} contains a brief revision of classical definitions in their algebraic reformulation. In \refsec{sec:potentials} we define potentials, study superharmonic elements and projections, and  obtain a version of the Riesz decomposition theorem,which is the starting point of our approach. Finally, we show that the set of superharmonic projections is a complete lattice. Transient projections are defined and investigated in \refsec{sec:transProj}. In particular, for Markov operators on the algebra $\BH$, we obtain various characterizations of such projections.

In \refsec{sec:rec-proj} we direct our attention to recurrent projections and show that superharmonic elements are fixed on recurrent parts. To emphasize the differences to the classical case, we also define skew recurrent projections, which coincide with positive recurrent projections if the considered algebra is commutative. Finally, we employ a theorem of F. Haag to study reformulations of classical criteria for positive recurrence. \refsec{sec:finite} examines the finite dimensional case, where further characterizations of transient projections are obtained.

In \refsec{sec:Projections} we apply the previous results to the study of idempotent Markov operators: our analysis of recurrent and transient projections leads to a structure theorem of such maps and allows to put the Choi-Effros product into more concrete terms.
\refsec{sec:poisson} deals with non-commutative Poisson boundaries and establishes a new characterization of weak* mean ergodic Markov operators, for which we obtain an abstract Poisson integral. We also show that such operators never have null recurrent projections.

Finally, in \refsec{sec:eb} we include an application to entanglement breaking channels. We show that for a Markov operator with Holevo representation any two projections in the fixed space commute,  and we determine the idempotent Markov operators which admit a Holevo representation.

\section{Notation and Preliminaries}\label{sec:prelim}

Throughout this paper $\hrH$ denotes a Hilbert space (with scalar product $\skal{\,\cdot\,}{\,\cdot\,}$ linear in the first component) and $\BH$ the algebra of all bounded linear operators on $\hrH$. If $x\in\BH$ is positive, we write $x\geq0$ and we denote by $\BH_+$ the set of all positive elements. The ideals of all trace class operators and all compact operators in $\BH$ are denoted by $\TH$ and $\KH$, respectively. For $x\in\TH$ the trace norm is given by $\trnorm{x}:=\tr(\sqrt{x^*x})$. The strong operator topology is referred to as \st. Whenever we discuss infinite sums in $\BH$, we consider convergence in this topology.
The rank one operators $t_\xi:\hrH\to\hrH:\eta\mapsto\skal{\eta}{\xi}\xi$ in $\BH$ will appear several times in our
discussions. We write $\Eins_\hrH$ or simply $\Eins$ for the identity operator on $\hrH$.

The set of all \emph{normal states} on the algebra $\BH$, i.e.\ the set of all linear functionals $\Fi$, for which there is a (unique) \emph{density operator} $\rho\in\TH$ with $\rho\geq 0$ and $\tr(\rho)=1$ such that $\Fi(x)=\tr(\rho x)$, is denoted by $\Zst$.

More generally, we consider a \emph{von Neumann algebra} $\cA\subseteq\BH$, i.e.\ a *-algebra of operators on $\hrH$, which is closed in the strong operator topology. We always assume $\Eins_\hrH\in\cA$. As for $\BH$, the cone of all $0\leq x\in\cA$ is denoted by $\cA_+$. By $\Zst(\cA)$ we denote the set of all \emph{normal states on $\cA$}, i.e.\ the set of all linear functionals $\Fi$ which can be implemented by a density operator $\rho\in\TH$, which is, however, not unique in this more general case. A normal state $\Fi\in\Zst(\cA)$ is \emph{faithful} if $\Fi(x^*x)=0$ implies $x=0$. The von Neumann algebra $\cA$ is called \emph{$\sigma$-finite} if there exists a faithful normal state on $\cA$.

For an orthogonal projection $p\in\cA$ its orthogonal complement $\Eins-p$ is denoted by $p^\bot$. If $\cP$ is any set of orthogonal projections in $\cA$ then $\bigvee\set{p\in\cP}$ stands for the supremum of this set, i.e.\ the smallest projection $q\in\cA$ such that $p\leq q$ for all $p\in\cP$. For a bounded increasing net $(a_i)_{i\in I}\subseteq\cA_+$ the supremum $\bigvee_{i\in I}a_i$ is defined likewise.
If $a$ is a self-adjoint element in a von Neumann algebra $\cA$ then $\supp a$ denotes its support projection, i.e.\ the smallest projection $p\in\cA$ such that $a=pap$. Similarly, for a normal state $\Fi$ on $\cA$ the support $\supp\Fi$ is the smallest projection $p\in\cA$ such that $\Fi(x)=\Fi(pxp)$ for all $x\in\cA$.

For the mathematical background of these notations we refer to \cite{mur} or \cite{tak}.

A map $T:\cA\to\cA$ is \emph{positive} if $T(\cA_+)\subseteq\cA_+$ and \emph{completely positive} if the map $\id_n\otimes T:M_n\otimes\cA\to M_n\otimes\cA$ is positive for every $n\in\NN$, where $\id_n$ is the identity map on $M_n$. If $T(\Eins)=\Eins$, it is called \emph{unital}. Such a map $T$ is \emph{normal} if $T\left(\bigvee_{i\in I}a_i\right)=\bigvee_{i\in I}T(a_i)$ for every bounded increasing net $(a_i)_{i\in I}\subseteq\cA_+$ or, equivalently, if $\Fi\circ T$ is normal for every normal state $\Fi$. In this case the map $\Zst(\cA)\ni\Fi\mapsto\Fi\circ T\in\Zst(\cA)$ is denoted by $T_*$ and called the \emph{pre-adjoint of $T$}. A normal completely positive unital map $T$ on $\cA$ is called a \emph{(quantum) Markov operator}. Its fixed space is denoted by $\fix T:=\Set{x\in\cA}{T(x)=x}$; a normal state $\Fi$ is called \emph{stationary} if $\Fi\circ T=\Fi$.

Finally, a positive element $a\in\cA_+$ is called \emph{subharmonic (w.r.t.\ $T$)} if $T(a)\geq a$; if, on the other hand, $T(a)\leq a$ then $a$ is called \emph{superharmonic}. It is elementary to see that the support of a normal \emph{stationary} state is subharmonic. Indeed, if $\Fi\circ T=\Fi\in\Zst(\cA)$ and $p:=\supp\Fi$ then $\Fi(T(p^\bot))=\Fi(p^\bot)=0$, hence $T(p^\bot)=p^\bot T(p^\bot)p^\bot\leq p^\bot$. The following observations on subharmonic projections are useful.

\begin{lem}\label{lem:charakSubharm}\emph{(\cite[Lem.\,2]{luc}, \cite[Thm.\,2]{uma})}
	\begin{enumerate}[label=(\alph*)]
		\item Let $\cA\subseteq\BH$ be a von Neumann algebra and $T:\cA\to\cA$ be a Markov operator. Then an orthogonal projection $p\in\cA$ is subharmonic, i.e.\ $T(p)\geq p$, if and only if\vspace*{-2.5ex}
			\begin{align*}
				p T(x)p = p T(p xp)p \qquad\mbox{for all $x\in\cA$.}  
			\end{align*}
		\item Let $p\in\cA$ be a subharmonic projection and $T_p:p\cA p\to p\cA p: x\mapsto p\,T(x)\,p$. Then ${T_p}^n(x)=p\,T^n(x)\,p$ for any $x\in p\cA p$ and $n\in\NN$.
	\end{enumerate}
\end{lem}

\subsection*{A Brief Revision of Classical Markov Chains}\label{sec:klassisch}

As a motivation for our approach we introduce classical notions of recurrence and transience and present them in a way that allows us to generalize them. For further details and the corresponding proofs we refer to \cite{rev} and \cite{dur}.

Let $\Omega$ be a discrete state space (finite or countable) and $T$ the \emph{transition matrix} of a (homogeneous) Markov chain on $\Omega$. In order to distinguish states in $\Omega$ from general states on operator algebras we prefer to call elements of $\Omega$ ``\emph{point-states}''.
The probability $t_{ij}^{(k)}$ to get from a point-state $i\in\Omega$ to $j\in\Omega$ in $k$ steps is equal to the $(i,j)$-th entry of $T^k$. The limit $\sum_{k=0}^\infty t_{jj}^{(k)}$ can be interpreted as the \emph{expected number of visits} of the point-state $j$ when the Markov chain starts in $j\in\Omega$.

Denoting the probability to ever reach a point-state $j$ when starting from $i$ by $\rho_{ij}$, we call a point-state $j\in\Omega$ \emph{transient} if $\rho_{jj}<1$ and \emph{recurrent} if $\rho_{jj}=1$. A point-state $j$ which can be reached from a recurrent point-state $i$ is itself recurrent, i.e.\ $\rho_{ii}=1$ and $\rho_{ij}>0$ imply $\rho_{jj}=1$. A recurrent point-state $j\in\Omega$ is called \emph{positive recurrent} if $\lim_{n\to\infty}\frac{1}{n}\sum_{k=1}^n t_{jj}^{(k)}>0$ and \emph{null recurrent} if $\lim_{n\to\infty}\frac{1}{n}\sum_{k=1}^n t_{jj}^{(k)}=0$. A subset $A\subseteq\Omega$ is called transient, recurrent, positive recurrent or null recurrent, respectively, if all $j\in A$ have the corresponding property.

\begin{thm}\label{thm:classical:transient}
A point-state $j\in\Omega$ is transient if and only if $\sum_{k=0}^\infty t_{jj}^{(k)}<\infty$, i.e.\ if the Markov chain is expected to hit $j$ only finitely many times.
\end{thm}

We want to reformulate this algebraically. Let $\linfty(\Omega)$ be the C$^*$-algebra of all uniformly bounded functions from $\Omega$ to $\CC$. Then every $f\in\linfty(\Omega)$ can be identified with a column vector and every finite measure on $\Omega$ (regarded as a row vector) is a linear functional on $\linfty(\Omega)$. The (normal) states $\Zst(\linfty(\Omega))$ correspond to the probability distributions on $\Omega$. Furthermore, for every point-state $i\in\Omega$ we obtain a state $\delta_i\in\Zst(\linfty(\Omega))$ with $\delta_i(\set{j})=\delta_{ij}$ (Kronecker delta) and each state $\psi\in\Zst(\linfty(\Omega))$ can uniquely be written as a (finite or infinite) convex combination of these $\delta_i$. If $\Fi\in\Zst(\linfty(\Omega))$ is \emph{stationary}, i.e.\ $\Fi\circ T=\Fi$, then every $i\in\Omega$ with $\Fi(\set{i})>0$ is positive recurrent. Hence we have the following

\begin{thm}\label{thm:classical:stationary}
	The support of a stationary state is positive recurrent.
\end{thm}

Let $\ChFkt_A\in\linfty(\Omega)$ be the characteristic function of $A\subseteq\Omega$ and set $\ChFkt_i:=\ChFkt_{\set{i}}$. Then $T^k(\ChFkt_j)(i)=\delta_i\circ T^k(\ChFkt_j)=t_{ij}^{(k)}$ and $T^k(\ChFkt_A)(i)$ describes the probability to hit $A$ in $k$ steps when  starting from the point-state $i\in\Omega$. We define $\cG(f):=\sum_{k=0}^\infty T^k(f)$ for $0\leq f\in\linfty(\Omega)$, where  $\cG(f)$ may also attain the value $+\infty$. Then we have $\cG(\ChFkt_i)(i)=\sum_{k=0}^\infty T^k(\ChFkt_i)(i)= \sum_{k=0}^\infty t_{ii}^{(k)}$, hence $\cG(\ChFkt_A)(i)$ can be interpreted as expected number of visits to $A\subseteq\Omega$ when starting from $i\in\Omega$ (be it finite or infinite).

If $i\in\Omega$ is recurrent and $f(i)>0$ for a positive function $f\in\linfty(\Omega)$ then $\cG(f)(i)\geq \cG(f(i)\cdot\ChFkt_i)(i)= f(i)\cdot\cG(\ChFkt_i)(i)=\infty$.
If $A\subseteq\Omega$ is transient and $\cG(\ChFkt_A)(j)>0$ then $j$ is transient, too, since $\rho_{ij}=0$ for all recurrent $i\in\Omega$ and all $j\in A$. Hence the support of $\cG(\ChFkt_A)$ is transient. If $A=\set{j}$, and thus for finite $A$, the function $\cG(\ChFkt_A)$ takes finite values only. In general, however, it may happen that $\cG(\ChFkt_A)$ attains the value $+\infty$. Nevertheless, there still exists a positive function $f\in\linfty(\Omega)$ such that $\cG(f)$ is bounded and has the same support as $\cG(\ChFkt_A)$. Indeed, by the (complete) maximum principle (\cite[Thm.\,2.1.12]{rev}) $\cG(\ChFkt_j)(i)\leq\cG(\ChFkt_j)(j)<\infty$ for all $j\in A$, $i\in\Omega$. Thus by a standard argument $f$ can be chosen as a suitable (infinite) weighted sum of $\ChFkt_j$, $j\in A$, (cf. the proof of \refthm{thm:potentialSupp}).
%
%
Hence we arrive at the following characterization of transient sets:

\begin{thm}\label{thm:classical:potential}
	A set $A\subseteq\Omega$ is transient if and only if there is a positive function $f\in\linfty(\Omega)$ with $\cG(f)$ finite such that $A$ is contained in the support of $\cG(f)$.
\end{thm}

These are the versions of (positive) recurrence and transience which we will generalize to the non-commutative context.


\section{Potentials and the Riesz Decomposition Theorem}\label{sec:potentials}

In order to define transient projections in \refsec{sec:transProj}, we introduce non-commutative potentials (cf.\ \cite{fr}, \cite{fag}, and \cite{uma}) as a generalization of the classical notion (see, e.g., \cite{rev} and \cite{dm}). We investigate their relation to superharmonic elements and projections and obtain a non-commutative version of the Riesz decomposition theorem, which will be a key tool in our discussion.

\begin{defn}\label{defn:potential}
	Let $T:\cA\to\cA$ be a Markov operator on a von Neumann algebra $\cA\subseteq\BH$.
	An element $x\in\cA_+$ is called \emph{$T$-summable} if $\sum_{n=0}^\infty{T^n(x)}$ exists in $\cA_+$.
	An element $y\in\cA_+$ is called \emph{potential (for $T$)} if there exists a $T$-summable element $x\in\cA_+$ such that $y=\sum_{n=0}^\infty{T^n(x)}$.
	By \[\cApot(T):=\Set{y\in\cA_+}{\exists\, x\in\cA_+ \mbox{ with } y=\sum\nolimits_{n=0}^\infty{T^n(x)}}\] (or simply $\cApot$ if no confusion can arise) we denote the set of all potentials for $T$.
\end{defn}

\begin{rem}\label{rem:potenialSuperharm}
	If $y=\sum_{n=0}^\infty{T^n(x)}$ is a potential then $y$ is superharmonic and $x=y-T(y)$; in particular, $x$ is uniquely determined. Indeed, since $T$ is normal and the net of partial sums is increasing, we have $y-T(y)=\sum_{n=0}^\infty{T^n(x)}-\sum_{n=1}^\infty{T^n(x)}=x\geq 0$. In this case $x\in\cA_+$ is also called the \emph{charge} of the potential $y\in\cApot$.
\end{rem}

In the next theorem we show that the classical Riesz decomposition theorem (cf.\ \cite[Thm.\,2.1.6]{rev} or \cite[no.\,IX.28]{dm}) and its proof carry over to the non-commuta\-tive situation.

\begin{thm}[\textbf{Riesz decomposition theorem}]\label{thm:RieszDecompThm}
	Let $T:\cA\to\cA$ be a Markov operator on a von Neumann algebra $\cA\subseteq\BH$.
	\begin{enumerate}[label=(\alph*)]
		\item\label{item:RieszDecompThm:potential} An element $y\in\cA_+$ is a potential if and only if $T(y)\leq y$ and $\stlim\limits_{n\to\infty}T^n(y)=0$.
		\item\label{item:RieszDecompThm:zerlegung} An element $a\in\cA_+$ is superharmonic if and only if there are elements $y\in\cApot$ and $0\leq h\in\fix T$ such that $a=y+h$. Such a decomposition is unique.
	\end{enumerate}
\end{thm}

\begin{proof}
\begin{description}[font=\normalfont\mdseries\em]
	\item[~\;\ref*{item:RieszDecompThm:potential}]
		If $y\in\cA_+$ is a potential then $y$ is superharmonic by \refrem{rem:potenialSuperharm} and from $y=\sum_{n=0}^\infty{T^n(x)}$ it follows that $T^k(y)=\sum_{n=k}^\infty{T^n(x)}\Geht{k\to \infty} 0$ \st.

		Conversely, for $x:=y-T(y)\geq 0$ we obtain
		\[
			\sum_{n=0}^N{T^n(x)}=\sum_{n=0}^N{T^n(y)}-\sum_{n=1}^{N+1}{T^n(y)} =y- T^{N+1}(y)\Geht{N\to \infty} y\ \ \st.
		\]
	\item[~\;\ref*{item:RieszDecompThm:zerlegung}]
		An element $a=y+h$ with $y$ and $h$ as above is superharmonic by \refrem{rem:potenialSuperharm}.

		Conversely, if $T(a)\leq a$ then $h:=\stlim\limits_{n\to\infty}T^n(a)$ exists and $0\leq h\in\fix T$. Set $y:=a-h$ then $T(y)\leq y$ and $\stlim\limits_{n\to\infty}T^n(y)=\stlim\limits_{n\to\infty}T^n(a)-h=0$. Hence $y$ is a potential by part \ref*{item:RieszDecompThm:potential}.\qedhere
\end{description}
\end{proof}

\begin{cor}\label{cor:potential_dominierter_fixpkt_verschwindet}
	Let $0\leq h\in\fix T$. If there is a potential $y\in\cApot$ such that $h\leq y$ then $h=0$.
\end{cor}

\begin{proof}
	Since $T$ is completely positive, we have $h=T^N(h)\leq T^N(y)\Geht{N\to\infty} 0 \ \ \st$.
\end{proof}

\begin{cor}\label{cor:cApotClosedCone}
	The set of potentials $\cApot\subseteq\cA_+$ for a Markov operator $T:\cA\to\cA$ is a norm-closed $T$-invariant cone.
\end{cor}

\begin{proof}
	It is immediate from \refthmX{thm:RieszDecompThm}{.\ref*{item:RieszDecompThm:potential}} that $\cApot$ is a $T$-invariant cone.
	Let $(y_j)_{j\in\NN}\subseteq\cApot$ be a sequence converging uniformly to $y\in\cA_+$. Then we have
	\[
		T(y)= T(\lim\nolimits_j y_j)=\lim\nolimits_j T(y_j)\leq\lim\nolimits_j y_j=y.
	\]
	For $\eps>0$ choose $j\in\NN$ with $\norm{y-y_j}<\frac{\eps}{2}$. Let $\xi\in\hrH$ with $\norm{\xi}\leq 1$ and $n_0\in\NN$ such that $\norm{T^n(y_j)\xi}<\frac{\eps}{2}$ for every $n\geq n_0$. Then we have
	\[
		\norm{T^n(y)\xi}\leq \norm{T^n(y-y_j)\xi}+\norm{T^n(y_j)\xi}< \norm{T^n}\,\norm{(y-y_j)}\,\norm{\xi}+\tfrac{\eps}{2}<\eps.
	\]
	Hence $y\in\cApot$ by \refthm{thm:RieszDecompThm}.\ref*{item:RieszDecompThm:potential}.
\end{proof}

\begin{rem}\label{rem:shift_auf_ell_infty}
	Consider $\cA=\ell^\infty(\NN)$ and $T_0$ the left shift on $\cA$, i.e.\ $T_0(f)(n)=f(n+1)$ for $f=(f(n))_{n\in\NN}$ in $\cA$. Then $\cApot$ is given by the positive decreasing sequences which converge to zero. Since $\Eins\notin\cApot$, the cone of potentials cannot be strongly closed.
	
	Furthermore, this example shows that the cone of charges does not need to be closed. Here it is given by the positive elements of $\ell^1(\NN)\subseteq\cA$.
\end{rem}

The following observation can be useful for the computation of potentials.

\begin{prop}\label{prop:potentialsForPowersOfT}
	Let $\cA\subseteq\BH$ be a von Neumann algebra, $T:\cA\to\cA$ a Markov operator, and $N\in\NN$. If $y\in\cA_+$ is a potential for $T^N$ then $\tilde y:=y+T(y)+\ldots T^{N-1}(y)$ is a potential for $T$.
\end{prop}
Note that if $x:=y-T^N(y)$ is the charge for $y$ with respect to $T^N$ then $x$ is also the charge for $\tilde y$ with respect to $T$.

\begin{proof}
	Let $y$ be a potential for $T^N$. By the characterization of potentials of the Riesz decomposition theorem~\refNumX{thm:RieszDecompThm}{.\ref*{item:RieszDecompThm:potential}} it follows that $y$ is superharmonic for $T^N$ and $\stlim\limits_{k\to\infty}(T^N)^k(y)=0$. This yields
	\[
		T(\tilde y) = T(y)+\ldots+T^{N-1}(y)+T^{N}(y) \leq T(y)+\ldots+T^{N-1}(y)+y=\tilde y,
	\]
	i.e.\ $\tilde y$ is superharmonic for $T$. Hence $T^n(y)\leq T^n(\tilde y)\leq\tilde y$ ($n\in\NN$), which implies that $\bigl(T^n(y)\bigr)_{n\in\NN}$ is bounded. Since the map $T$ is $\sigma\mbox{-}\mathrm{stop}$-continuous, it is $\mathrm{stop}$-continuous on bounded sets. Therefore, the stop-limit of the sequence $\bigl(T^m(T^{Nk}(y))\bigr)_{k\in\NN}$ is zero for each $0\leq m\leq N-1$. Hence $\bigl(T^n(y)\bigr)_{n\in\NN}$ can be decomposed into $N$ disjoined subsequences that are \st-convergent to zero. This implies $\stlim\limits_{n\to\infty}T^n(y)=0$.
\end{proof}

\begin{lem}\label{lem:TraegerBleibenGleich}
	Let $T:\cA\to\cA$ be a Markov operator on a von Neumann algebra $\cA\subseteq\BH$ and let $x\in\cA_+$ with support 
	$p:=\supp x$.
	Then $\supp T(x) = \supp T(p)$.\\
	In particular, if $\supp x = \supp y$ for $x,y\in\cA_+$ then $\supp T(x) = \supp T(y)$.
\end{lem}

\begin{prf}
	Since $\supp x = \supp (\alpha\cdot x)$ for every $\alpha>0$, we can assume that $0\leq x\leq\Eins$. Let $p:=\supp x$, then $p\geq x$. Hence $T(p)\geq T(x)$ and $\supp T(p)\geq \supp T(x)$.

	Conversely, using the spectral theorem let $p_k:=\ChFkt_{]\frac{1}{k},\norm{x}]}(x)$ for every $k\in\NN$. Then each $p_k$ is an orthogonal projection and $(p_k)_{k\in\NN}$ converges monotonically from below to $\supp x=p$. Now $p_k\leq k\cdot x$ implies $T(p_k)\leq k\cdot T(x)$ and $T(p_k) \leq \supp T(x)$. Letting $k\to\infty$ we have $T(p)\leq \supp T(x)$, since $T$ is normal. Hence $\supp T(p)\leq\supp T(x)$.
\end{prf}

\begin{prop}\label{prop:suppEbenfallsSuperharm}
	If  $T:\cA\to\cA$ is a Markov operator and $a\in\cA_+$ is superharmonic, i.e.\ $T(a)\leq a$, then its support projection is superharmonic, too.\end{prop}

\begin{prf}
	Let $p:=\supp a$, then we have $\supp T(p)=\supp T(a)\leq p$ by \reflem{lem:TraegerBleibenGleich}. Since $\norm{T(p)}\leq 1$, it follows that $T(p)\leq\supp T(p)\leq p$.
\end{prf}

In particular, support projections of potentials are superharmonic. This special case is implicitly contained in 
\cite[Prop.\,4]{fr}.

In general, the support of a potential is not in $\cApot$ itself. Compare, however, \refprop{prop:ptr_potential_falls_T_wme} and the subsequent remarks.

\begin{thm}\label{thm:superharm_complete_lattice} \emph{\cite{rz}}\ \ 
	Let $T:\cA\to\cA$ be a Markov operator on a von Neumann algebra $\cA\subseteq\BH$. Then the set of superharmonic projections is a complete lattice.
\end{thm}

Since $p$ is subharmonic if $p^\bot$ is superharmonic, it follows that the set of subharmonic projections is a complete lattice, too.

Infima of superharmonic projections are easily seen to be superharmonic (see, e.g., \cite{luc}). Their suprema are superharmonic, too, as was recently shown by Raggio and Zangara by considering faces of normal states (cf.\ \cite{rz}). 
In the present approach this appears as an easy consequence of the previous proposition.

\begin{prf}
	We have to show that for any family $(p_i)_{i\in I}\subseteq\cA$ of superharmonic projections the supremum $p_\vee:=\bigvee_{i\in I}p_i$ and the infimum $p_\wedge:=\bigwedge_{i\in I}p_i$ are both superharmonic. 

	Since $T(p_\wedge)\leq T(p_i)\leq p_i$ for all $i\in I$, it follows that $T(p_\wedge)\leq p_\wedge$.

	Let $J\subseteq I$ be a finite subset and $p_J:=\bigvee_{i\in J}p_i=\supp\sum_{i\in J}p_i$. Then $T(\sum_{i\in J}p_i)=\sum_{i\in J}T(p_i)\leq\sum_{i\in J}p_i$ and by \refprop{prop:suppEbenfallsSuperharm} we have $T(p_J)\leq p_J$.

	The finite subsets of $I$ are directed by inclusion and, obviously, $(p_J)_{J\subseteq I \text{ finite}}$ is an increasing net of orthogonal projections such that $p_\vee=\bigvee_{i\in I}p_i=\bigvee_{J\subseteq I \text{ finite}} p_J$. Hence $(T(p_J))_{J\subseteq I \text{ finite}}$ is a bounded increasing net in $\cA_+$ and since $T$ is normal,
	\[
		\textstyle T(p_\vee)=T\bigl(\bigvee_J p_J\bigr)=\bigvee_J T(p_J) \leq \bigvee_J p_J = p_\vee.\qedhere
	\]
\end{prf}

\section{Transient Projections}\label{sec:transProj}

In this section we use potentials to characterize the transient part of an arbitrary von Neumann algebra $\cA\subseteq\BH$ with respect to a Markov operator $T:\cA\to\cA$. The second part of this section concentrates on transience on $\BH$.

Our notion of transience relies on the following observations.

\begin{lem}\label{lem:PotentialTraegerApprox}
	Let $T:\cA\to\cA$ be a Markov operator on a von Neumann algebra $\cA\subseteq\BH$ and let $y\in\cApot$.
	\begin{enumerate}[label=(\alph*)]
		\item\label{item:PotentialTraegerApprox:posElt} There exists a $T$-summable element $\tilde x\in\cA_+$ such that $\supp \tilde x = \supp y$.
		\item\label{item:PotentialTraegerApprox:projections} There is an increasing sequence of $T$-summable orthogonal projections $(p_m)_{m\in\NN}$ such that $\bigvee_{m\in\NN}\,p_m = \supp y$.
	\end{enumerate}
\end{lem}

A version of this lemma for continuous semigroups can be deduced from \cite{uma}. Since their proof uses the resolvent, we give a proof which is adapted to our discrete situation.

\begin{prf}
	\begin{description}[font=\normalfont\mdseries\em]
		\item[~\;\ref*{item:PotentialTraegerApprox:posElt}]
			Let $x\in\cA_+$ be the charge of $y$. Fix $0<\lambda<1$ and define $x_\lambda:=\sum_{k=0}^\infty{\lambda^k\, T^k(x)}\in\cA_+$ then $\supp x_\lambda = \supp y$. Since $T$ is normal, we obtain
			\begin{align*}
				\sum_{n=0}^\infty T^n(x_\lambda) & = \sum_{n=0}^\infty T^n\biggl(\sum_{k=0}^\infty{\lambda^k\, T^k(x)}\biggr)
				= \sum_{k=0}^\infty\lambda^k\,T^k \biggl(\sum_{n=0}^\infty{T^n(x)}\biggr)\\
				& = \sum_{k=0}^\infty\lambda^k\,T^k (y)\; \leq\; \sum_{k=0}^\infty\lambda^k\,y\;=\;\frac{1}{1-\lambda}\;y\ \in\;\cA_+.
			\end{align*}
		\item[~\;\ref*{item:PotentialTraegerApprox:projections}]
			Let $p_m:=\ChFkt_{]\frac{1}{m},\norm{x_\lambda}]}(x_\lambda)$ for every $m\in\NN$. Since $p_m\leq m\cdot x_\lambda$, we have $\sum_{n=0}^\infty{T^n(p_m)}\in\cA_+$ for all $m\in\NN$ and $\bigvee_{m\in\NN}\,p_m = \supp x_\lambda = \supp y$.\qedhere
	\end{description}
\end{prf}

\begin{thm}\label{thm:trans}
	Let $\cA\subseteq\BH$ be a von Neumann algebra, $T:\cA\to\cA$ a Markov operator, and $p\in\cA$ an orthogonal projection. Then the following statements are equivalent:
	\begin{enumerate}[label=(\roman*)]
		\item\label{item:trans:potentials} There is a family $(y_i)_{i\in I}\subseteq\cApot$ of potentials for $T$ such that $p\leq\bigvee\limits_{i\in I}\supp y_i$.
		\item\label{item:trans:positive} There is a family $(x_i)_{i\in I}\subseteq\cA_+$ of $T$-summable elements such that $p\leq\bigvee\limits_{i\in I}\supp x_i$.
		\item\label{item:trans:projections} There is a family $(p_j)_{j\in J}\subseteq\cA$ of $T$-summable orthogonal projections such that $p\leq\bigvee\limits_{j\in J} p_j$.
	\end{enumerate}
\end{thm}

\begin{proof}
	\begin{description}[font=\normalfont\mdseries\em]
		\item[\ref*{item:trans:potentials}$\folgt$\ref*{item:trans:positive}:]
			This is immediate from \reflem{lem:PotentialTraegerApprox}.
		\item[\ref*{item:trans:positive}$\folgt$\ref*{item:trans:projections}:]
			Let $J:=I\times\NN$. For each $j=(i,m)\in J$ define $p_j:=\ChFkt_{]\frac{1}{m},\norm{x_i}]}(x_i)$ as in the proof of  \reflem{lem:PotentialTraegerApprox}.\ref{item:PotentialTraegerApprox:projections}. Then $\sum_{n=0}^\infty{T^n(p_j)}\in\cA_+$ for all $j\in J$ and $p\leq\bigvee_{j\in J} p_j$.
		\item[\ref*{item:trans:projections}$\folgt$\ref*{item:trans:potentials}:]
			For each $j\in J$ define $y_j:=\sum_{n=0}^\infty{T^n(p_j)}$. Then each $y_j$ is a potential and $y_j\geq p_j$ and thus $\supp y_j\geq p_j$ for all $j\in J$.\qedhere
	\end{description}
\end{proof}

\begin{defn}\label{defn:trans}
	An orthogonal projection $p\in\cA$ is \emph{transient (w.r.t.\ $T$)} if it satisfies the equivalent conditions of \refthm{thm:trans}.
	We call the supremum of all transient projections the \emph{maximal transient projection} and denote it by $\ptr(T)$  (or $\ptr$ for short).
\end{defn}

If $p\in\cA$ is a $T$-summable orthogonal projection then $p$ is obviously transient. However, there are examples (cf.\ \refrem{rem:shift_auf_ell_infty}) where $\Eins\in\cA$ is transient, too. But, clearly, the unit of $\cA$ cannot be $T$-summable.

Clearly, the transient projections form a complete lattice. In particular, $\ptr$ is also transient and can be written as
\[
	\ptr=\bigvee\Set{\supp y}{y\in\cApot}=\bigvee\Set{p\in\cA}{\mbox{$p=p^*=p^2$ is $T$-summable}}.
\]
	
\begin{rem}\label{rem:trans}
	By \refthm{thm:superharm_complete_lattice} the maximal transient projection $\ptr$ is superharmonic. Hence the algebra $\cAtr:=\ptr\,\cA\,\ptr$ is invariant under $T$, i.e.\ $T(\cAtr)\subseteq\cAtr$, and $T\vert_{\cAtr}$ is \emph{submarkovian}, i.e.\ a completely positive normal map such that $T(\Eins)\leq\Eins$.
	Moreover, it is shown in \cite[Thm.\,8]{uma} that if $\cA$ is $\sigma$-finite then $\ptr$ is the support of a potential.
\end{rem}

In general a superharmonic projection does not need to be transient. Indeed, $\Eins\in\cA$ is always superharmonic, since $T(\Eins)=\Eins$. But $\Eins$ is not transient whenever there is a stationary normal state $\Fi\in\Zst(\cA)$ as we will see in the next section.

\begin{prop}\label{prop:trans_and_superharm_in_sonderfaellen}
	Let $T:\cA\to\cA$ be Markov operator on a von Neumann algebra $\cA\subseteq\BH$.
	\begin{enumerate}[label=(\alph*)]
		\item\label{item:trans_superharm:ergodic} If $T$ is ergodic, i.e.\ $\fix T=\CC\cdot\Eins$, then every superharmonic projection $p\in\cA$ with $0\neq p\neq\Eins$ is transient.
		\item\label{item:trans_superharm:ptr=0} For the maximal transient projection $\ptr$ we have $\ptr=0$ if and only if every superharmonic element $a\in\cA_+$ is in $\fix T$.
	\end{enumerate}
\end{prop}

\begin{proof}
	\begin{description}[font=\normalfont\mdseries\em]
		\item[~\;\ref*{item:trans_superharm:ergodic}]
			Since $p$ is superharmonic, by the Riesz decomposition theorem \refKlammern{thm:RieszDecompThm} there is a potential $y\in\cApot$ and an element $0\leq h\in\fix T$ such that $p=y+h$. Since $p\lneqq\Eins$ and $T$ is ergodic, this yields $h=0$. Hence $p$ is a potential and thus transient.
		\item[~\;\ref*{item:trans_superharm:ptr=0}]
			Again $a$ can be written as a sum of a potential $y$ and a positive fixed point $h$ of $T$. If $\ptr=0$, there are no non-zero potentials. Hence we have $y=0$ and $a=h\in\fix T$. For the converse, note that by \refcor{cor:potential_dominierter_fixpkt_verschwindet} every potential in $\fix T$ vanishes.\qedhere
	\end{description}
\end{proof}

Note that if $b\in\cA_+$ with $\norm{b}\leq 1$ is subharmonic then $(\Eins-b)\geq 0$ is superharmonic. Hence if $\Fi\in\Zst(\cA)$ is a stationary normal state for an \emph{ergodic} Markov operator $T$ and $p:=\supp\Fi\neq\Eins$ then $p^\bot$ is transient. Furthermore, part \ref*{item:trans_superharm:ptr=0} implies that if $\ptr=0$ then also every subharmonic $b\in\cA_+$ is in $\fix T$.


For the rest of this section we restrict ourselves to the case $\cA:=\BH$. This allows us to provide additional characterizations of transient projections and support projections of potentials.  For this we need the following elementary observation, for which we couldn't find a suitable reference.

\begin{lem}\label{lem:SummeSkalarProdEndlForallEta}
	Let $\cA\subseteq\BH$ be a von Neumann algebra. If $(a_n)_{n\in\NN_0}\subseteq\cA_+$ is a sequence such that $\sum_{n=0}^\infty{\skal{a_n\eta}{\eta}} < \infty$ for every $\eta\in\hrH$ then $\sum_{n=0}^\infty{a_n}$ exists in $\cA_+$.
\end{lem}

\begin{prf}
	Set $s_N:=\sum_{n=0}^N{a_n}\in\cA_+$. Then for $\eta\in\hrH$ we have:
	\[
		\norm{s_N^\frac{1}{2}\eta}^2 = \skal{s_N^\frac{1}{2}\eta}{s_N^\frac{1}{2}\eta} = \skal{s_N\eta}{\eta} = \textstyle\sum_{n=0}^N{\skal{a_n\eta}{\eta}}.
	\]
	Obviously, $s_N\leq s_M$ if $N\leq M$. Hence $s_N^\frac{1}{2}\leq s_M^\frac{1}{2}$, since the square root is operator monotone. Consequently, for all $\eta\in\hrH$ we have:
	\[
		\sup_{N\in\NN}\norm{s_N^\frac{1}{2}\eta}^2 = \lim_{N\to\infty}\skal{\sum_{n=0}^N{a_n}\eta}{\eta} = \sum_{n=0}^\infty{\skal{a_n\eta}{\eta}} < \infty.
	\]

	Applying the Principle of Uniform Boundedness yields $\sup_{N\in\NN}\norm{s_N^\frac{1}{2}}<\infty$. Hence $(s_N^\frac{1}{2})_N$ is a bounded and increasing sequence of elements in $\cA_+$, which converges strongly to an element $b\in\cA_+$. Therefore, the limit $\sum_{n=0}^\infty{a_n}= \stlim_N s_N= \stlim_N\left(s_N^\frac{1}{2}\; s_N^\frac{1}{2}\right)=b^2$ exists by \st-continuity of multiplication on bounded sets.
\end{prf}

\begin{prop}\label{prop:transBH}
	Let $T:\BH\to\BH$ be a Markov operator and $p\in\BH$ an orthogonal projection. Then the following statements are equivalent:
	\begin{enumerate}[label=(\roman*)]
		\item\label{item:transBH:trans}
			$p$ is transient.
		\item\label{item:transBH:tXiSummierbar}
			There is a family of unit vectors $(\xi_j)_{j\in J}\subseteq\hrH$ such that $p\hrH\subseteq\clSPAN{\xi_j}{j\in J}$ and each rank one operator $t_{\xi_j}$ is $T$-summable.
		\item\label{item:transBH:tXiEtaEtaSummierbar}
			There is a family of unit vectors $(\xi_j)_{j\in J}\subseteq\hrH$ such that $p\hrH\subseteq\clSPAN{\xi_j}{j\in J}$ and   $\sum\limits_{n=0}^\infty{\skal{T^n(t_{\xi_j})\eta}{\eta}}<\infty$ for all $j\in J$ and $\eta\in \hrH$.
	\end{enumerate}
\end{prop}

Note that condition \ref*{item:transBH:tXiEtaEtaSummierbar} is quite close to one of the classical characterizations for transience mentioned in \refsec{sec:klassisch}. For a stronger condition in the finite dimensional case compare \refthm{thm:trans_endl}.

\begin{prf}
	\begin{description}[font=\normalfont\mdseries\em]
		\item[\ref*{item:transBH:trans} $\folgt$ \ref*{item:transBH:tXiSummierbar}:]
			Since $p$ is transient, there exists a family of $T$-summable orthogonal projections $(p_i)_{i\in I}\subseteq\BH$ such that $p\leq \bigvee_{i\in I}p_i$.

			For each $i\in I$ choose an orthogonal basis $(\xi_j)_{j\in J_i}$ for $p_i\hrH$. Set $J:=\bigcup_{i\in I} J_i$ then for each $j\in J$ there is an $i\in I$ with $j\in J_i$ and $\sum_{n=0}^\infty{T^n(t_{\xi_j})} \leq\sum_{n=0}^\infty{T^n(p_i)}\in\BH_+$. Additionally,  $\clSPAN{\xi_j}{j\in J}=\left(\bigvee_{i\in I}p_i\right)\hrH\supseteq p\hrH$.
		\item[\ref*{item:transBH:tXiSummierbar} $\folgt$ \ref*{item:transBH:trans}:]
			Since $p\hrH\subseteq\clSPAN{\xi_j}{j\in J}=\left(\bigvee_{j\in J}t_{\xi_j}\right)\hrH$, we have $p\leq\bigvee_{j\in J}t_{\xi_j}$. Hence $p$ is transient by \refthm{thm:trans}.
			\item[\ref*{item:transBH:tXiSummierbar} $\eqvt$ \ref*{item:transBH:tXiEtaEtaSummierbar}:]
			Direct consequence of \reflem{lem:SummeSkalarProdEndlForallEta}.\qedhere
	\end{description}
\end{prf}

In the following theorem we characterize support projections of potentials with \emph{separable range}. Clearly, these orthogonal projections are transient and by \refprop{prop:suppEbenfallsSuperharm} they are superharmonic.

\begin{thm}\label{thm:potentialSupp}
	Let $T:\BH\to\BH$ be a Markov operator and $q\in\BH$ an orthogonal projection such that $q\hrH$ is \emph{separable}. If $q$ is \emph{superharmonic} then the following statements are equivalent:
	\begin{enumerate}[label=(\roman*)]
		\item\label{item:potentialSupp:PotentialTraeger}
			There exists a potential $y\in\cApot$ such that $q=\supp y$.
		\item\label{item:potentialSupp:tXiSummierbar}
			There is an orthonormal basis $(\xi_j)_{j\in J}$ for $q\hrH$ such that each rank one operator $t_{\xi_j}$ is $T$-summable.
		\item\label{item:potentialSupp:tXiEtaEtaSummierbar}
			There is an orthonormal basis $(\xi_j)_{j\in J}$ for $q\hrH$ such that $\sum\limits_{n=0}^\infty{\skal{T^n(t_{\xi_j})\eta_0}{\eta_0}}<\infty$ for all $j\in J$ and $\eta_0\in q\hrH$.
	\end{enumerate}
\end{thm}

The implications \ref*{item:potentialSupp:PotentialTraeger} $\folgt$ \ref*{item:potentialSupp:tXiSummierbar} $\eqvt$ \ref*{item:potentialSupp:tXiEtaEtaSummierbar} hold without the assumption of $q\hrH$ being separable.

\begin{prf}
	\begin{description}[font=\normalfont\mdseries\em]
		\item[\ref*{item:potentialSupp:PotentialTraeger} $\folgt$ \ref*{item:potentialSupp:tXiSummierbar}:]
			By \reflem{lem:PotentialTraegerApprox} there is an increasing sequence $(q_m)_{m\in\NN}$ of $T$-sum\-ma\-ble orthogonal projections such that $\bigvee_{m\in\NN} q_m = q$. For each $m\geq 1$ choose an orthonormal basis $(\xi_i)_{i\in I_{m+1}}$ for $(q_{m+1}-q_m)\hrH\subseteq q_{m+1}\hrH$ and $(\xi_i)_{i\in I_1}$ for $q_1\hrH$. Let $I:=\bigcup_{m\in\NN}I_m$ then $(\xi_i)_{i\in I}$ is an orthonormal basis for $q\hrH$. Since each $t_{\xi_i}\leq q_m$ for some $m\in\NN$ and $\sum_{n=0}^\infty{T^n(q_m)}$ exists in $\cA_+$, the same holds true for $\sum_{n=0}^\infty{T^n(t_{\xi_i})}$.

		\item[\ref*{item:potentialSupp:tXiSummierbar} $\folgt$ \ref*{item:potentialSupp:tXiEtaEtaSummierbar}:]
			Trivial.

		\item[\ref*{item:potentialSupp:tXiEtaEtaSummierbar} $\folgt$ \ref*{item:potentialSupp:PotentialTraeger}:]
			Let $(\xi_j)_{j\in J}$ be an orthonormal basis for $q\hrH$ having the above properties. Then $J$ is finite or countable (since $q\hrH$ is separable) and we can assume that $J=\{1,2,\ldots,m\}$ or $J=\{1,2,3,\ldots\}$, respectively.

			Since $T(q)\leq q$, we have $T^n(t_{\xi_j}) \leq T^n(q)\leq q$ and $T^n(t_{\xi_j})\eta_1= 0$ for all $\eta_1\in (q\hrH)^\bot$. We can decompose each $\eta\in\hrH$ into $\eta=\eta_0+\eta_1$ with $\eta_0\in q\hrH$ and $\eta_1\in (q\hrH)^\bot$ and thus have $\skal{T^n(t_{\xi_j})\eta}{\eta} = \skal{T^n(t_{\xi_j})\eta_0}{\eta_0}.$
			This yields $\sum_{n=0}^\infty\skal{T^n(t_{\xi_j})\eta}{\eta} = \sum_{n=0}^\infty\skal{T^n(t_{\xi_j})\eta_0}{\eta_0} < \infty$ for all $j\in J$ and $\eta=\eta_0+\eta_1\in\hrH$. Hence $\sum_{n=0}^\infty{T^n(t_{\xi_j})}\in\BH_+$ for every $j\in J$ by \reflem{lem:SummeSkalarProdEndlForallEta}.

			Hence defining $y_j:=\sum_{n=0}^\infty{T^n(t_{\xi_j})}$ we can find a sequence $(a_j)_{j\in J}\subseteq\; ]0,1[\,$ such that $a_j\norm{y_j} < \left(\frac{1}{2}\right)^j$ for every $j\in J$. Thus for all $\eta\in\hrH$ with $\norm{\eta}\leq 1$ we have:
			\[
				\sum_{j\in J}{a_j \sum_{n=0}^\infty{\skal{T^n(t_{\xi_j})\eta}{\eta}}}
				= \sum_{j\in J}{a_j {\skal{y_j\,\eta}{\eta}}}
				\leq \sum_{j\in J}{a_j \norm{y_j}}
				\leq \sum_{j\in J}{\left(\tfrac{1}{2}\right)^j} < \infty.
			\]

			This allows us to swap the above limits and $T$ being normal, we obtain for all $\eta\in\hrH$:\vspace*{-2ex}
			\begin{align*}
				\hspace{-1ex}\infty > \sum_{j\in J}{a_j \sum_{n=0}^\infty{\skal{T^n(t_{\xi_j})\eta}{\eta}}} = 
				\sum_{n=0}^\infty{\skal{T^n\bigl(\textstyle\sum_{j\in J} {a_j\,t_{\xi_j}}\bigr)\eta}{\eta}}.
			\end{align*}
			If we define $x:=\sum_{j\in J}{a_j t_{\xi_j}}$ then $x\geq 0$ and $q=\supp x$. Applying \reflem{lem:SummeSkalarProdEndlForallEta} again we see that $y:=\sum_{n=0}^\infty{T^n(x)}$ exists in $\BH_+$. Since $T^0(x)=x$, we have $q=\supp x \leq \supp y$.
			On the other hand, since $q$ is superharmonic, we have $\supp y=\supp\bigl(\sum_{n=0}^\infty{T^n(x)}\bigr)\leq q$.\qedhere
	\end{description}
\end{prf}

In addition to \refprop{prop:transBH} and \refthm{thm:potentialSupp} we find the following sufficient conditions for an arbitrary orthogonal projection $p\in\BH$ to be transient.

\begin{cor}\label{cor:TransCond} Let $T:\BH\to\BH$ be a Markov operator and $p\in\BH$ an orthogonal projection. The following are sufficient conditions for $p$ to be transient:
	\begin{enumerate}[label=(\roman*)]
		\item\label{item:TransCond:XiEtaInH}
			For all $\xi\in p\hrH$ the rank one operator $t_\xi$ is $T$-summable.
		\item\label{item:TransCond:XiEtaInH_ONB}
			There exists an orthonormal basis $(\xi_i)_{i\in I}$ for $p\hrH$ such that each rank one operator $t_{\xi_j}$ is $T$-summable.
	\end{enumerate}
\end{cor}

\begin{prf}
	Obviously, \ref*{item:TransCond:XiEtaInH_ONB} is a direct consequence of \ref*{item:TransCond:XiEtaInH}. If condition \ref*{item:TransCond:XiEtaInH_ONB} is satisfied $p$ is transient by \refNumX{prop:transBH}{.\ref*{item:transBH:tXiSummierbar}}, since $p\hrH=\clSPAN{\xi_i}{i\in I}$.
\end{prf}

The following example disproves some further possible conjectures.

\begin{exmp}\label{exmp:transCond}
	Consider the Hilbert space $\hrH:=\Lzwei([0,2\pi]
	)$ and let $e_n\in\hrH$ be given by  $e_n(s):=e^{ins}$, $s\in[0,2\pi]$. Then $(e_n)_{n\in\ZZ}$ is an orthonormal basis for $\hrH$. For $g\in\Linfty([0,2\pi])$ we define the multiplication operator $M_g$ on $\hrH$ by $M_{g}(f):= gf$. Clearly, one obtains $(M_{e_1})^n=M_{e_n}$, $M_{e_n}^*=M_{e_{-n}}$, and $M_{e_{-n}}\,t_{e_k}\,M_{e_n} = t_{e_{k-n}}$.
	
	Now let $T:\BH\to\BH:a\mapsto M_{e_1}^*\,a\,M_{e_1}$ and $x:=\sum_{k=1}^\infty{2^{-k}\,t_{e_k}}$. Then we have
	\begin{align*}
		\sum_{n=0}^\infty{T^n(x)} & = \sum_{n=0}^\infty{(M_{e_1}^*)^n\,x\,M_{e_1}^n} =  \sum_{n=0}^\infty{\sum_{k=1}^\infty{2^{-k}\,M_{e_{-n}}\,t_{e_k}\,M_{e_n}}} \\
		& = \sum_{k=1}^\infty{\sum_{n=0}^\infty{2^{-k}\,M_{e_{-n}}\,t_{e_k}\,M_{e_n}}} =  \sum_{k=1}^\infty{2^{-k}\,\sum_{n=0}^\infty{t_{e_{k-n}}}}.
	\end{align*}
	Since $\norm{\sum_{n=0}^\infty{t_{e_{k-n}}}}=1$ for every $k\in\NN$, this sum converges to a bounded operator $y$ on $\hrH$. Note that $\supp y\geq \supp x=\sum_{k=1}^\infty{t_{e_k}}$ and $\supp y\geq \sum_{n=0}^\infty{t_{e_{1-n}}}$. Hence $\supp y = \Eins = \ptr$ and thus every projection is transient.
	Let $f,g\in\hrH=\Lzwei([0,2\pi])$ be real-valued functions. Then for every $N\in\NN$ we have
	\begin{align*}
		\sum_{n=0}^N{ \skal{T^n(t_f)\,g}{g}} & = \sum_{n=0}^N{\skal{M_{e_n}^*\,t_f\,M_{e_n}\,g}{g}} = \sum_{n=0}^N{\skal{t_f\,M_{e_n}\,g}{M_{e_n}\,g}} \\
		& = \sum_{n=0}^N{\skal{\skal{M_{e_n}\,g}{f}\,f}{M_{e_n}\,g}} = \sum_{n=0}^N{\skal{e_n\,g}{f}\skal{f}{e_n\,g}} \\
		& = \sum_{n=0}^N{\abs{\skal{f}{e_n\,g}}^2} = \sum_{n=0}^N{\abs{\skal{fg}{e_n}}^2} = \sum_{n=0}^N{\abs{\skal{fg}{e_{-n}}}^2},
	\end{align*}
	since $\skal{fg}{e_n}=\overline{\skal{\overline{fg}}{\overline{\vphantom{f}e_n}}}=\overline{\skal{fg}{e_{-n}}}$ for all $n\in\ZZ$. Hence $\sum_{n=0}^N{\abs{\skal{fg}{e_n}}^2}\geq \frac12 \sum_{n=-N}^N{\abs{\skal{fg}{e_n}}^2}$ for every $N\in\NN$.
	Furthermore, $\widehat{fg}(n):=\skal{fg}{e_n}$ is the $n$th Fourier coefficient of $fg$. 
	
	Define $f,g\in\hrH
	$ by $f(s)=s^{-1/8}$ and $g(s)=s^{-3/8}$, $s\in[0,2\pi]$. Then $fg\in\Leins([0,2\pi])$ but $fg\notin\Lzwei([0,2\pi])$. It follows that $\sum_{n=0}^N{ \skal{T^n(t_f)\,g}{g}} = \sum_{n=0}^N{\abs{\skal{fg}{e_n}}^2}\Geht{N\to\infty}\infty$. On the other hand, $f^2\in\Lzwei([0,2\pi])$ and hence $\sum_{n=0}^\infty{\skal{T^n(t_f)\,f}{f}}<\infty$.

	This example shows:
	\begin{enumerate}[label=(\arabic*)]
		\item If $y\in\cApot$ and $q:=\supp y$ such that $q\hrH$ is separable then by \refthm{thm:potentialSupp} there is always an orthonormal basis $(\xi_i)_{i\in I}$ of $q\hrH$ such that each $t_{\xi_i}$ is $T$-sum\-ma\-ble, but this does not exclude the existence of a vector $\xi\in q\hrH$ for which $t_\xi$ is not $T$-summable. In particular, Condition \ref*{item:TransCond:XiEtaInH} of \refcor{cor:TransCond} is no necessary condition.
		\item Unlike the classical situation it might happen that there is no vector $\xi\neq0$ in the range of a transient projection $p$ such that the corresponding rank one operator $t_\xi$ is $T$-summable: In the above example $p:=\frac{1}{\norm{f}^2}\,t_f$ is a transient projection with range $p\hrH=\Set{\lambda f}{\lambda\in\CC}$ but $t_f$ is not $T$-summable. In particular, even Condition \ref*{item:TransCond:XiEtaInH_ONB} of \refcor{cor:TransCond} is not a necessary condition.
		\item We also see that $\sum_{n=0}^\infty{ \skal{T^n(t_\xi)\,\xi}{\xi}}<\infty$ for $\xi\in\hrH$ does not imply that $t_\xi$ is $T$-summable, equivalently $\sum_{n=0}^\infty{\skal{T^n(t_\xi)\,\eta}{\eta}}<\infty$ for all $\eta\in\hrH$.
	\end{enumerate}

\end{exmp}

\section{Recurrent Projections}\label{sec:rec-proj}

After treating transient projections, we examine the recurrent part of the algebra now.

\begin{defn}\label{defn:recurrent}
	Let $\cA\subseteq\BH$ be a von Neumann algebra and $T:\cA\to\cA$ a Markov operator. We call an orthogonal projection  $p\in\cA$
	\begin{itemize}
		\item \emph{recurrent (w.r.t.\ $T$)} if $p\leq \ptr(T)^\bot=\Eins - \ptr(T)$,
		\item \emph{positive recurrent (w.r.t.\ $T$)} if there is a family of stationary normal states $(\Fi_i)_{i\in I}$ on $\cA$ such that $p\leq\bigvee_{i\in I}(\supp\Fi_i)$,
		\item \emph{skew recurrent (w.r.t.\ $T$)} if there is a family of stationary normal states $(\psi_i)_{i\in I}$ on $\cA$ which is faithful on $p\cA p$.
	\end{itemize}
	
	Furthermore, we define the \emph{maximal positive recurrent projection} as the supremum of all positive recurrent projections and denote it by  $\ppr(T)$  (or $\ppr$ for short).
\end{defn}

For commutative $\cA$ the notions of positive recurrence and skew recurrence coincide. However, this is no longer true for general non-commutative algebras (cf.\ \refexmp{exmp:pos_rekur}).

Clearly, the positive recurrent projections form a complete lattice. In particular, $\ppr$ is positive recurrent, too, and can be written as:
\[
	\ppr=\bigvee\Set{\supp \Fi}{\Fi\circ T=\Fi\in\Zst(\cA)}.
\]

\begin{rem}\label{rem:pRsubharm}
	By \refthm{thm:superharm_complete_lattice} the maximal positive recurrent projection $\ppr$ is subharmonic, i.e.\ $T(\ppr)\geq\ppr$.
\end{rem}

\begin{prop}\label{prop:uma05_rec_orth_trans}\emph{\cite{uma}}
	\begin{enumerate}[label=(\alph*)]
		\item\label{item:uma05_r_t:ppr_rec} The maximal positive recurrent projection $\ppr$ is recurrent, i.e.\ $\ppr\leq\ptr^\bot$.
		\item\label{item:uma05_r_t:sigma-finite} If $\cA$ is $\sigma$-finite then $\ppr$ is the support of a stationary normal state $\Fi\in\Zst(\cA)$.
	\end{enumerate}
\end{prop}
For convenience we include the short proof.
\begin{prf}
	\begin{description}[font=\normalfont\mdseries\em]
		\item[~\;\ref*{item:uma05_r_t:ppr_rec}] Let $y\in\cApot$ be a potential with charge $x:=y-T(y)$ and $\Fi\in\Zst(\cA)$ a stationary normal state, then
		\[
			\Fi(y)=\Fi\biggl(\sum_{n=0}^\infty T^n(x)\biggr)=\sum_{n=0}^\infty\Fi\bigl(T^n(x)\bigr)=\sum_{n=0}^\infty\Fi(x).
		\]
		Therefore, $\Fi(x)=0=\Fi(y)$ which implies $\supp y\leq(\supp\Fi)^\bot$. Since this holds for every stationary normal state  $\Fi\in\Zst(\cA)$ and since  $\ptr$ coincides with the supremum of the support projections of all potentials $y\in\cApot$, it follows that $\ptr\leq\ppr^\bot$.
		\item[~\;\ref*{item:uma05_r_t:sigma-finite}] By definition there is a family $(\Fi_i)_{i\in I}\subseteq\Zst(\cA)$ of stationary normal states such that $\ppr=\bigvee_{i\in I}(\supp\Fi_i)$. Choose a well-ordering $\preceq$ on $I$ and for $j\in I$ set $q_j:=\bigvee_{i\preceq j}(\supp\Fi_i)-\bigvee_{i\precneqq j}(\supp\Fi_i)$. 
		Then $(q_j)_{j\in I}$ is a family of mutually orthogonal projections.
		Since $\cA$ is $\sigma$-finite, the set $I_0:=\Set{j\in I}{q_j\neq0}\subseteq I$ is at most countable and $\ppr=\sum_{j\in I}q_j=\sum_{j\in I_0}q_j=\bigvee_{i\in I_0}(\supp\Fi_i)$. If $I$ is finite, we set $\Fi:=\frac{1}{\abs{I}\vphantom{^2}}\sum_{i\in I}\Fi_i$. Otherwise, we identify $I_0$ with $\NN$ and define $\Fi:=\sum_{i\in\NN}2^{-i}\,\Fi_i$ to obtain the desired state.\qedhere
	\end{description}
\end{prf}

From \refNumX{prop:uma05_rec_orth_trans}{.\ref*{item:uma05_r_t:ppr_rec}} follows that every positive recurrent projection is also recurrent and an orthogonal projection $\cA\ni p\leq\ppr^\bot$ cannot be skew recurrent.

\begin{defn}\label{defn:nullrecurrent}
	We call $\pnr:=\Eins-(\ptr+\ppr)$ the \emph{maximal null recurrent projection} and every orthogonal projection $p\leq\pnr$ is said to be \emph{null recurrent}.
\end{defn}

In the classical setting, by the so called Hopf decomposition, superharmonic elements are fixed on recurrent parts (see, e.g., \cite[Thm.\,3.1.3]{kre}). In the non-commutative situation an analogous statement clearly holds on positive recurrent parts. The following theorem shows that this remains true on the entire recurrent part.

\begin{thm}\label{thm:recurrence}
	Let $\cA\subseteq\BH$ be a von Neumann algebra, $T:\cA\to\cA$ a Markov operator, and $p\in\cA$ an orthogonal projection. Then the following statements are equivalent:
	\begin{enumerate}[label=(\roman*)]
		\item\label{item:rec:rec} $p$ is recurrent, i.e.\ $p\leq\ptr^\bot$.
		\item\label{item:rec:superharm-fixed} For all $a\in\cA_+$ with $T(a)\leq a$ we have
			\[
				pT(a)p=pap.
			\]
	\end{enumerate}
\end{thm}

\begin{prf}
	First, let $p\leq\ptr^\bot$ and let $a\in\cA_+$ be superharmonic. By the Riesz decomposition theorem \refKlammern{thm:RieszDecompThm} there are $y\in\cApot$ and $0\leq h\in\fix T$ such that $a=y+h$. Since $\supp y\leq\ptr$, we have
	\[
		pT(a)p=pT(y+h)p=pT(y)p+pT(h)p=php=pap.
	\]
	Conversely, assume that $p\not\leq\ptr^\bot$. Then $p\,\ptr\, p\neq0$. Since $\ptr$ is transient, by \refthm{thm:trans} there is a family of $T$-summable elements $(x_i)_{i\in I}\subseteq\cA_+$ such that $\ptr\leq\bigvee_{i\in I}\supp x_i$. Therefore,
	\[
		0\neq p\,\ptr\, p\leq p\left(\bigvee\nolimits_{\!\!i\in I}\supp x_i\right) p = \bigvee\nolimits_{\!\!i\in I}(p\,\supp x_i\,p)
	\]
	and there is an $i_0\in I$ with $p x_{i_0} p\neq0$. Let $y:=\sum_{n=0}^\infty T^n(x_{i_0})$. Then $y$ is superharmonic and 
	\[
		pyp-pT(y)p=p(y-T(y))p=p x_{i_0} p\neq0.\qedhere
	\]
\end{prf}

\begin{rem}\label{rem:recurrent}
	If $\ptr=0$ and $p$ is the support of a stationary normal state then $T(p)=p$. This can also be deduced from \refpropX{prop:trans_and_superharm_in_sonderfaellen}{.\ref*{item:trans_superharm:ptr=0}}.
\end{rem}

Now we investigate the relation between skew recurrence and (positive) recurrence.

\begin{prop}\label{prop:skew_rek_charak} Let $T:\cA\to\cA$ be a Markov operator.
\begin{enumerate}[label=(\alph*)]
	\item\label{item:skew_rek_charak:equiv} For an orthogonal projection $p\in\cA$ the following statements are equivalent:
		\begin{enumerate}[label=(\roman*)]
			\item\label{item:skew_rek_charak:equiv:defn} $p$ is skew recurrent.
			\item\label{item:skew_rek_charak:equiv:mitppr} 
				$p\wedge\ppr^\bot=0$.
		\end{enumerate}
	\item\label{item:skew_rek_charak:pos=>skew} Every positive recurrent projection is also skew recurrent.
\end{enumerate}
\end{prop}

\begin{prf}
	\begin{description}[font=\normalfont\mdseries\em]
		\item[\ref{item:skew_rek_charak:equiv:defn}$\folgt$\ref{item:skew_rek_charak:equiv:mitppr}:]
			Let $q:= p\wedge\ppr^\bot$ and $(\psi_i)_{i\in I}\in\Zst(\cA)$ as in \refdefn{defn:recurrent}. Then $q\in p\cA p$ and $\psi_i(q)=\psi_i(\ppr q\ppr)=0$ for all $i\in I$. Since $(\psi_i)_{i\in I}$ is faithful on $p\cA p$, it follows that $q=0$.
		\item[\ref{item:skew_rek_charak:equiv:mitppr}$\folgt$\ref{item:skew_rek_charak:equiv:defn}:]
			Suppose $p$ is not skew recurrent. Then there exists an orthogonal projection $0\neq r\leq p$ such that $\Fi(r)=0$ for all stationary normal states $\Fi\in\Zst(\cA)$, i.e.\ $r\leq(\supp\Fi)^\bot$. Hence $r\leq\ppr^\bot$.
		\item[~\;\ref{item:skew_rek_charak:pos=>skew}]
			Let $p$ be positive recurrent. Then there is a family $(\Fi_i)_{i\in I}\in\Zst(\cA)$ of stationary normal states such that $p\leq\bigvee_{i\in I}(\supp\Fi_i)=:q$. Since $(\Fi_i)_{i\in I}$ is faithful on $q\cA q$, it is faithful on $p\cA p$.\qedhere
	\end{description}
\end{prf}

The next example shows, that the converse of \ref{item:skew_rek_charak:pos=>skew} does not hold: There are skew recurrent projections which are not (positive) recurrent.

\begin{exmp}\label{exmp:pos_rekur}
	Let $\cA:=M_2(\CC)$ and let $(e_{ij})_{i,j\in\set{1,2}}$ be the canonical system of matrix units. Then $\Fi:=\tr(e_{22} \; \cdot\;)$ is a normal state and $T: M_2(\CC)\to M_2(\CC):\,x\mapsto e_{21}^* x e_{21} + e_{22}^* x e_{22} = \Fi(x)\Eins$ is completely positive with $T(\Eins)=\Eins$.
	
	If we set $p:=e_{11}$ then $T(p)=0$ and $\sum_{n=0}^\infty{T^n(p)}=p$. Thus $p$ is transient.
	Furthermore, $\Fi$ is stationary. Hence $\supp\Fi=e_{22}$ is positive recurrent and we have $\ptr=e_{11}$ and $\ppr=\ptr^\bot=e_{22}$.
	
	Now consider the rank one projection $q:=\frac{1}{2}\left(\begin{smallmatrix}1&1\\1&1\end{smallmatrix}\right)$ and let $x=\sum_{i,j}x_{ij}\,e_{ij}\in M_2(\CC)$. Then $qxq=\frac{1}{4}(x_{11}+x_{12}+x_{21}+x_{22})\left(\begin{smallmatrix}1&1\\1&1\end{smallmatrix}\right)$ and if $\Fi(qxq) = \frac{1}{4}(x_{11}+x_{12}+x_{21}+x_{22}) = 0$, it follows that $qxq=0$, too. Therefore, $q$ is skew recurrent. But, clearly, $q\nleq \ptr^\bot = e_{22}$ and thus $q$ is not recurrent.
\end{exmp}

The compact operators $\KH$ on $\hrH$ are a two sided ideal. Therefore, any Markov operator $T$ with a finite Kraus decomposition $T(x)=\sum_{i=1}^m{a_i^* x a_i}$ leaves this ideal invariant.
For such operators we obtain another characterization of skew recurrence, which is equivalent to a well-known classical characterization of positive recurrence. In order to prove this, we need the following result, which also provides a criterion for the existence of a stationary normal state. It was proven by Haag in his diploma theses \cite{flo}.  Since it is not publicly available, we also include its proof.

\begin{thm}\label{thm:FloRecurrent}\emph{\cite[Thm.\,2.1.6]{flo}}\ \
	Let $T:\BH\to\BH$ be a Markov operator such that $T(\KH)\subseteq\KH$ and let $\xi\in\hrH$. Then the following statements are equivalent:
	\begin{enumerate}[label=(\roman*)]
 		\item\label{item:FloRec:Cesaro} $\limsup\limits_{N\to\infty} \displaystyle\frac{1}{N} \displaystyle\sum\limits_{n=0}^{N-1}{\skal{T^n(t_\xi)\xi}{\xi}}\;\gneqq\;0$.
		\item\label{item:FloRec:skew} There exists a stationary normal state $\Fi$ such that $\Fi(t_\xi)\neq 0$.
	\end{enumerate}
\end{thm}

\begin{prf}
	Let $\norm{\xi}=1$ and define $\Fi_\xi:=\skal{\,\cdot\,\xi}{\xi}=\tr(t_\xi\,\cdot\,)\in\Zst$. Then for $a\in\BH$ we have $\Fi_\xi(T^n(a))=\skal{T^n(a)\xi}{\xi} = \tr\left(t_\xi\,T^n(a)\right) = \tr\left(T_*^n(t_\xi)\,a\right)$, where $T_*:\TH\to\TH$ is the pre-adjoint of $T$ regarded as a map on the trace class operators.
	Let $S_N:=\frac{1}{N}\sum_{n=0}^{N-1}{T_*^n}$ and $\sigma_N:=S_N(t_\xi)$ for $N\in\NN$. Then $\sigma_N\in\TH$ and $\tr(\sigma_N\,t_\xi)=\frac{1}{N}\sum_{n=0}^{N-1}{\tr\left(T_*^n(t_\xi)\,t_\xi\right)}=\frac{1}{N}\sum_{n=0}^{N-1}{\skal{T^n(t_\xi)\xi}{\xi}}$.
	
	\begin{description}[font=\normalfont\mdseries\em]
		\item[\ref*{item:FloRec:Cesaro} $\folgt$ \ref*{item:FloRec:skew}:]
			We can assume that $(\tr(\sigma_N\,t_\xi))_{N\in\NN}\subseteq\RR$ is convergent (otherwise we could use a convergent subsequence). Hence by assumption we have $\alpha:=\lim_{N\to\infty}\tr(\sigma_N\,t_\xi)>0$.
			
			Since $t_\xi\in\TH_1:=\Set{t\in\TH}{\trnorm{t}\leq 1}$ and $T_*$ is a contraction, we have $\sigma_N=S_N(t_\xi)\in\TH_1$ for all $N\in\NN$. The set $\TH_1$ is $\sigma(\TH,\KH)$-compact.   
			Hence there is a subnet $(\sigma_{N_\lambda})_{\lambda\in\Lambda}$ of $(\sigma_N)_{N\in\NN}\subseteq\TH_1$ and an element $\tilde{\ro}\in\TH_1$ such that $\tr(\tilde{\ro}\,c)=\lim_{\lambda}\tr(\sigma_{N_\lambda}c)$ for all compact operators $c\in\KH$.
			
			For every $\lambda\in\Lambda$ the operator $\sigma_{N_\lambda}$ is positive and by our assumption $\tr(\tilde{\ro}\,t_\xi)=\lim_\lambda\tr(\sigma_{N_\lambda}\,t_\xi)=\alpha>0$. Therefore, we have $\tilde{\ro}\gneqq0$ and hence $\Fi:=\tr(\ro\,\cdot\,)$ with $\ro:=\frac{\tilde{\ro}}{\trnorm{\tilde{\ro}}}$ is a normal state and $\Fi(t_\xi)=\frac{\alpha}{\trnorm{\tilde{\ro}}}>0$. Furthermore, for all $c\in\KH$ we conclude from the compactness of $T(c)$ that
			\begin{align*}
				\tr(T_*(\tilde{\ro})\,c) & = \tr(\tilde{\ro}\;T(c)) =\lim\nolimits_{\lambda}\tr(\sigma_{N_\lambda}\,T(c)) =\lim\nolimits_{\lambda}\tr(T_*(\sigma_{N_\lambda})\;c) \\
				& =\lim\nolimits_{\lambda}\tr\left(\bigl[\sigma_{N_\lambda}+\tfrac{1}{N_\lambda}\bigl(T_*^{N_\lambda}(t_\xi)-t_\xi\bigr)\bigr]\;c \right) = \tr(\tilde{\ro}\,c). \tag*{$(\ast)$} 
			\end{align*}
			This implies $T_*(\ro)=\ro$, since $\KH$ is separating for $\TH$. Hence $\Fi$ is a stationary normal state with $\Fi(t_\xi)\neq0$.
		\item[\ref*{item:FloRec:skew} $\folgt$ \ref*{item:FloRec:Cesaro}:]
			Assume that $\lim_{N\to\infty}\frac{1}{N}\sum_{n=0}^{N-1}{\skal{T^n(t_\xi)\xi}{\xi}}=0$. 
			Let $\Fi=\tr(\ro\,\cdot\,)$ be a stationary normal state with $\Fi(t_\xi)=\tr(\ro\,t_\xi)\neq 0$. Since $T_*(\ro)=\ro$ and $t_\xi$ is an orthogonal projection,
			\begin{align*}
				\tr(t_\xi\,\ro) = \tr(t_\xi\,S_N(\ro))
				=\, & \tr\bigl(t_\xi\,S_N(t_\xi\ro t_\xi)\bigr) + \tr\bigl(t_\xi\,S_N(t_\xi^\bot\ro t_\xi)\bigr)\\
				& + \tr\bigl(t_\xi\,S_N(t_\xi\ro t_\xi^\bot)\bigr) + \tr\bigl(t_\xi\,S_N(t_\xi^\bot\ro t_\xi^\bot)\bigr).
			\end{align*}
			Using the Cauchy-Schwarz inequality we obtain that 
			\[ 
				\bigl|\tr\bigl(t_\xi\,S_N(t_\xi^\bot\ro t_\xi)\bigr)\bigr|^2=\bigl|\Fi\bigl(t_\xi S_N^*(t_\xi)t_\xi^\bot\bigr)\bigr|^2=\bigl|\Fi\bigl(((S_N^*(t_\xi))^\frac12 t_\xi)^* ((S_N^*(t_\xi))^\frac12 t_\xi^\bot)\bigr)\bigr|^2
			\]
			and similarly $\bigl|\tr\bigl(t_\xi\,S_N(t_\xi\ro t_\xi^\bot)\bigr)\bigr|^2$ are both majorized by
			\[
				\bigl|\Fi\bigl(t_\xi\,S_N^*(t_\xi)\,t_\xi\bigr)\cdot \Fi\bigl(t_\xi^\bot\,S_N^*(t_\xi)\,t_\xi^\bot\bigr)\bigr| = \bigl|\tr\bigl(t_\xi\,S_N(t_\xi\ro t_\xi)\bigr)\cdot \tr\bigl(t_\xi\,S_N(t_\xi^\bot\ro t_\xi^\bot)\bigr)\bigr|.
			\]
			Using that $t_\xi x t_\xi=\Fi_\xi(x)t_\xi$ for every $x\in\BH$, we see that $\tr\bigl(t_\xi\,S_N(t_\xi\ro t_\xi)\bigr)=\frac{\Fi_\xi(\ro)}{N}\sum_{n=0}^{N-1}{\skal{T^n(t_\xi)\xi}{\xi}}\TO 0$. Hence $\tr(t_\xi\,\ro) = \lim_{N\to\infty}\tr\bigl(t_\xi\,S_N(t_\xi^\bot\ro t_\xi^\bot)\bigr)$.
			
			Defining $\ro_1:= t_\xi^\bot\ro t_\xi^\bot$ there is an $n_1\in\NN$ such that $\tr\bigl(t_\xi\,T_*^{n_1}(\ro_1)\bigr)\geq \frac12 \tr(t_\xi\,\ro)$ and similarly to Equation $(\ast)$ it follows that
			\[
				\lim_{N\to\infty}\tr\bigl(t_\xi\,S_N(\ro_1)\bigr)= \lim_{N\to\infty}\tr\bigl(t_\xi\,S_N(T_*^{n_1}(\ro_1))\bigr).
			\]
			Hence $\tr(t_\xi \ro)=\lim_{N\to\infty}\tr\bigl(t_\xi\,S_N(T_*^{n_1}(\ro_1))\bigr)$. But now we can repeat the steps above for $T_*^{n_1}(\ro_1)$ to see that $\tr(t_\xi\,\ro) = \lim_{N\to\infty}\tr\bigl(t_\xi\,S_N(t_\xi^\bot T_*^{n_1}(\ro_1) t_\xi^\bot)\bigr)$. If we define $\ro_2:= t_\xi^\bot T_*^{n_1}(\ro_1) t_\xi^\bot$, we obtain an $n_2\in\NN$ such that $\tr\bigl(t_\xi\,T_*^{n_2}(\ro_2)\bigr)\geq \frac12 \tr(t_\xi\,\ro)$. Continuing in this fashion, we obtain a recursively defined sequence $(\ro_i)_{i\in\NN}$ with $\ro_{i+1}:=t_\xi^\bot T_*^{n_i}(\ro_i) t_\xi^\bot$ for all $i\in\NN$ such that $\tr\bigl(t_\xi\,T_*^{n_i}(\ro_{i})\bigr)\geq \frac12 \tr(t_\xi\,\ro)$ and $\tr(t_\xi\,\ro) = \lim_{N\to\infty}\tr\bigl(t_\xi\,S_N(\ro_i)\bigr)$.
			
			But since $T(\Eins)=\Eins$ and $\trnorm{x}= 
			\tr(x)$ for every $0\leq x\in\TH$, we have
			\begin{align*}
				\trnorm{\ro_i} &=\tr(\ro_i) = \tr(\ro_i\; T^{n_i}(\Eins)) = \tr\bigl(T_*^{n_i}(\ro_i)\bigr) =\tr\bigl(t_\xi\, T_*^{n_{i}}(\ro_i)\bigr) + \tr\bigl(t_\xi^\bot\, T_*^{n_{i}}(\ro_i)\bigr)\\
				&\geq \tfrac12 \tr(t_\xi\,\ro) + \tr\bigl(t_\xi^\bot T_*^{n_{i}}(\ro_i)t_\xi^\bot\bigr)=\tfrac12 \tr(t_\xi\,\ro) + \trnorm{\ro_{i+1}}.
			\end{align*}
			Hence $0\leq\trnorm{\ro_{i+1}}\leq \trnorm{\ro_i}-\tfrac12 \tr(t_\xi\,\ro)$ for all $i\in\NN$. But this is not possible.\qedhere
	\end{description}
\end{prf}

We note that we did not make use of the invariance of the set of compact operators in the proof of the implication \ref*{item:FloRec:skew} $\folgt$ \ref*{item:FloRec:Cesaro}.


\begin{cor}\label{cor:equivskewrecurrent}
	Let $T:\BH\to\BH$ be a Markov operator on $\BH$ that satisfies $T(\KH)\subseteq\KH$.
	\begin{enumerate}[label=(\alph*)]
		\item\label{item:equivskewrec:ppr_bot} Let $\xi\in\hrH$. Then $\lim\limits_{N\to\infty} \frac{1}{N}\sum_{n=0}^{N-1}{\skal{T^n(t_\xi)\xi}{\xi}}=0$ if and only if $t_\xi\leq\ppr^\bot$.  
		\item\label{item:equivskewrec:skew} For an orthogonal projection $p\in\BH$ the following statements are equivalent:
		\begin{enumerate}[label=(\roman*)]
			\item\label{item:equivskew:skew} $p$ is skew recurrent, i.e.\ there exists a family of stationary normal states $(\psi_i)_{i\in I}\subseteq\Zst(\cA)$ which is faithful on $p\BH p$.
			\item\label{item:equivskew:Flo} For all $\xi\in p\hrH$ we have: \quad $\limsup\limits_{N\to\infty} \displaystyle\frac{1}{N} \displaystyle\sum\limits_{n=0}^{N-1}{\skal{T^n(t_\xi)\xi}{\xi}}\;\gneqq\;0$.
		\end{enumerate}
	\end{enumerate}
\end{cor}

\begin{prf}
	\begin{description}[font=\normalfont\mdseries\em]
		\item[\ref*{item:equivskew:skew} $\folgt$ \ref*{item:equivskew:Flo}:]
			Let $\xi\in p\hrH$, then $0\leq t_\xi\in p\BH p$. Hence there is an $i\in I$ such that $\psi_i(t_\xi)>0$ and \refthm{thm:FloRecurrent} implies $\limsup_{N\to\infty} \frac{1}{N} \sum_{n=0}^{N-1}{\skal{T^n(t_\xi)\xi}{\xi}}>0$.
		\item[\ref*{item:equivskew:Flo} $\folgt$ \ref*{item:equivskew:skew}:]
			Suppose $p$ is not skew recurrent. Then $r:=p\wedge\ppr^\bot\neq0$ by \refprop{prop:skew_rek_charak}. Let $\xi\in r\hrH$ then $\Fi(t_\xi)\leq\Fi(\ppr^\bot)=0$ for every stationary normal state $\Fi\in\Zst(\cA)$. Applying \refthm{thm:FloRecurrent} again completes the proof.\qedhere
	\end{description}
\end{prf}

We have already seen that skew recurrent projections are not necessarily recurrent. This is different to the classical situation, where \refcorX{cor:equivskewrecurrent}{.\ref*{item:equivskew:Flo}} provides a criterion for positive recurrence.

As an illustration of our notions of positive recurrence and transience we discuss as an example a *-automorphism on $\BH$.

\begin{exmp}\label{exmp:automBH}
	Let $T:\BH\to\BH$ be a *-automorphism. Then there is a unitary $u\in\BH$ such that $T(x)=u^*xu$ for all $x\in\BH$ and we denote by $\sigma(u)\subseteq\TT:=\Set{\zeta\in\CC}{\abs{\zeta}=1}\subseteq\CC$ the spectrum of $u$.
	
	For simplicity we assume that $u$ admits a cyclic vector $\xi\in\hrH$. It induces a spectral measure $\mu$ on $\TT$, which is in fact supported by $\sigma(u)$.
	By the Spectral Theorem we may identify $\hrH$ with $\Lzwei(\TT,\mu)$, $u$ with the multiplication operator $M_{\id} 
	$ which is given by $(M_{\id} f)(z):=z\cdot f(z)$ for $f\in\Lzwei(\TT,\mu)$, $z\in\TT$, and $\xi$ with the constant function $\TT\ni z\mapsto1$ in $\Lzwei(\TT,\mu)$.
	The von Neumann subalgebra of $\BH$ generated by $u$, as well as its commutant, may both be identified with $\Linfty(\TT,\mu)$, where a function $g\in\Linfty(\TT,\mu)$ is identified with the corresponding multiplication operator on $\Lzwei(\TT,\mu)$ (cf.\ also \refexmp{exmp:transCond}).
	
	The spectral measure $\mu$ can be uniquely decomposed into $\mu=\mu_\mathrm{pp}+\mu_\mathrm{ac}+\mu_\mathrm{sc}$, where $\mu_\mathrm{pp}$, $\mu_\mathrm{ac}$, and $\mu_\mathrm{sc}$ denote its pure point, absolutely continuous, and singular continuous part, respectively. Correspondingly, there are three mutually orthogonal projections $p_\mathrm{pp}$, $p_\mathrm{ac}$, and $p_\mathrm{sc}$ in $\Linfty(\TT,\mu)\subseteq\BH$ such that $p_\mathrm{pp}+p_\mathrm{ac}+p_\mathrm{sc}=\Eins$ and $p_\mathrm{pp}u$, $p_\mathrm{ac}u$, and $p_\mathrm{sc}u$ have pure point, absolutely continuous, and singular continuous spectrum, respectively.
	
	Clearly, $T$ admits a stationary normal state $\Fi$ if and only if its corresponding density operator $\rho\in\TH$ commutes with $u$, hence is contained in $\Linfty(\TT,\mu)$. Since $\rho$ is a compact operator, this implies $\rho\in p_\mathrm{pp}\;\Linfty(\TT,\mu)$. It easily follows that the maximal positive recurrent projection $\ppr$ of $T$ is given by $p_\mathrm{pp}$.
	
	On the other hand, we show that whenever $p_\mathrm{ac}=\Eins$ then for the maximal transient projection $\ptr$ we have $\ptr=\Eins$ (hence, in general, we have $p_\mathrm{ac}\leq\ptr$):
	Let $\mu$ be absolutely continuous with respect to the Lebesgue measure $\lambda$ on $\TT$. Then $\mu$ has a density $f\in\Leins(\TT,\lambda)$ with respect to $\lambda$.
	For any $l\in\NN$ let $A_l:=\Set{z\in\TT}{\abs{f(z)}\leq l}$ and $\ChFkt_l:=\ChFkt_{A_l}$.
	If $g\in\Lzwei(\TT,\mu)$ then
	\[
		\int_\TT \abs{\ChFkt_lfg}^2\,\mathrm{d}\lambda = \int_\TT \ChFkt_lf\abs{g}^2f\,\mathrm{d}\lambda\leq l\cdot\int_\TT \abs{g}^2\,\mathrm{d}\mu<\infty,
	\]
	hence $\ChFkt_lfg\in\Lzwei(\TT,\lambda)$.
	For $k\in\ZZ$ we define the function $e_k$ on $\TT$ by $e_k(z):=z^k$ for $z\in\TT$ and we let $e_{k,l}:=\ChFkt_le_k$ ($k\in\ZZ$, $l\in\NN$).
	Then $(M_{\id}^*)^n\,e_k=e_{k-n}$ and for every $g\in\Lzwei(\TT,\mu)$ we have
	\begin{align*}
		\skal{T^n(t_{e_{k,l}})g}{g}_{\Lzwei(\mu)} &= \skal{t_{e_{k,l}} M_{\id}^n\,g}{M_{\id}^n\,g}_{\Lzwei(\mu)} =\abs{\skal{M_{\id}^n\,g}{e_{k,l}}_{\Lzwei(\mu)}}^2\\
		&= \abs{\skal{\ChFkt_l\,g}{e_{k-n}}_{\Lzwei(\mu)}}^2 = \abs{\int_\TT\ChFkt_l\,g\;\overline{e_{k-n}}\,\mathrm{d}\mu}^2 \\
		&=\abs{\int_\TT\ChFkt_l\,g\;\overline{e_{k-n}}\;f\,\mathrm{d}\lambda}^2 = \abs{\skal{\ChFkt_lfg}{e_{k-n}}_{\Lzwei(\lambda)}}^2.
	\end{align*}
	Since $\ChFkt_lfg\in\Lzwei(\TT,\lambda)$ and $\skal{\ChFkt_lfg}{e_m}$ is the $m$th Fourier coefficient of $\ChFkt_lfg$ we obtain for all $k\in\ZZ$, $l\in\NN$ and all $g\in\Lzwei(\TT,\mu)$:
	\[
		\sum_{n=0}^\infty \skal{T^n(t_{e_{k,l}})g}{g}_{\Lzwei(\mu)} = \sum_{n=0}^\infty \abs{\skal{\ChFkt_lfg}{e_{k-n}}_{\Lzwei(\lambda)}}^2 \leq\norm{\ChFkt_l f g}_{\Lzwei(\lambda)}^2 <\infty\;.
	\]
	Hence each $t_{e_{k,l}}$ is $T$-summable by \reflem{lem:SummeSkalarProdEndlForallEta}.
	Finally, since $\lim_{l\to\infty} e_{k,l}=e_k$ for all $k\in\ZZ$, the set $\Set{e_{k,l}}{k\in\ZZ,\;l\in\NN}$ is total in $\hrH$ and \refprop{prop:transBH} implies $\ptr=\Eins$.

\end{exmp}

\section{The Finite Dimensional Case}\label{sec:finite}

In the finite dimensional case the Markov-Kakutani Theorem ensures the existence of a stationary state for any Markov operator. Therefore, we have $\ppr\neq0$, hence $\ptr\neq\Eins$. \reflem{lem:pTr_summierbar_endl} below shows that $\ppr^\bot$ is $T$-summable. This means that $\ptr=\ppr^\bot$ (cf.\ the remark after \refthm{thm:trans}) and thus $\sum_{n=0}^\infty{T^n(\ptr)}\in\cA_+$. This also implies $\pnr=\Eins-(\ptr+\ppr)=0$, i.e.\ there is no null recurrent part.

A version of the next result is also contained in \cite{uma} as Lemma 7. However, the proof there seems to lack the final argument.

\begin{lem}\label{lem:pTr_summierbar_endl}
	Let $\cA\subseteq\BH$ be a finite dimensional von Neumann algebra and $T:\cA\to\cA$ a Markov operator. Then  $\ppr^\bot$ is $T$-summable.
\end{lem}

\begin{prf}
	Since $\ppr$ is subharmonic, we have $T(\ppr^\bot)\leq\ppr^\bot$. Hence $\bigl(T^n(\ppr^\bot)\bigr)_{n\in\NN}$ is a bounded decreasing sequence in $\ppr^\bot\cA \ppr^\bot$ and thus convergent to an element $x_0\in\fix T$.

	Since $T_*$ is a contraction and $\cA$ is finite dimensional, $\Fi_\omega:=\lim_{N\to\infty}\frac{1}{N}\sum_{n=0}^{N-1}{\omega\circ T^n}$ exists for all $\omega\in\Zst(\cA)$ and gives stationary state.
	Since $x_0\in\ppr^\bot\cA\ppr^\bot$, we have $\Fi_\omega(x_0)=0$. Therefore, for all $\omega\in\Zst(\cA)$ we have:
	\[
		\omega(x_0) = 
		\lim_{N\to\infty}\frac{1}{N}\sum_{n=0}^{N-1}{\omega\bigl(T^n(x_0)\bigr)} = \Fi_\omega(x_0) =0.
	\]
	Hence $\bigl(T^n(\ppr^\bot)\bigr)_{n\in\NN}$ is (norm-)convergent to zero and there is an $n_0\in\NN$ such that $\norm{T^{n_0}(\ppr^\bot)} \leq \frac12\norm{\ppr^\bot}$. This implies $T^{n_0}(\ppr^\bot) \leq \frac12\,\ppr^\bot$ and thus $T^{k\cdot n_0}(\ppr^\bot) \leq 2^{-k}\,\ppr^\bot$. Therefore, we have $\sum_{n=k\cdot n_0}^{(k+1)n_0}{T^n(\ppr^\bot)}\leq n_0\,2^{-k}\,\ppr^\bot$ and $\sum_{n=0}^\infty{T^n(\ppr^\bot)}\leq 2n_0\,\ppr^\bot$.
\end{prf}

In particular, in the finite dimensional setting $\lim_{n\to\infty}T^n(\ppr)=\Eins$ always holds.

From the previous lemma we can deduce several equivalent characterizations for transient projections.

\begin{thm}\label{thm:trans_endl}
	Let $\cA\subseteq\BH$ be a finite dimensional von Neumann algebra and $T:\cA\to\cA$ a Markov operator. Then for an orthogonal projection $p\in\cA$ the following statements are equivalent:
	\begin{enumerate}[label=(\arabic*)]
		\item\label{item:trans_endl:trans}
			$p$ is transient.
		\item\label{item:trans_endl:Summierbar}
			$p$ is $T$-summable, i.e. $\sum\limits_{n=0}^\infty{T^n(p)}\in\cA_+$.
		\item\label{item:trans_endl:PotenzenGegenNull}
			$\lim\limits_{n\to\infty}T^n(p)= 0$.
		\item\label{item:trans_endl:KeinInvarianterZustand}
			For every stationary state $\Fi\in\Zst(\cA)$ we have $\Fi(p)=0$.
		\item\label{item:trans_endl:PotentialTraeger}
			There is an element $y\in\cApot$ such that $p\leq \supp y$.
	\end{enumerate}
	
	If $\cA=\BH$ and $(\xi_i)_{i\in I}$ an orthonormal basis for $p\hrH$ then these statements are equivalent to the following ones:
	\begin{enumerate}[label=(\arabic*), resume]
		\item\label{item:trans_endl:t_XiSummierbar}
			For all $i\in I$ we have \ $\sum\limits_{n=0}^\infty{T^n(t_{\xi_i})}\in\cA_+$.
		\item\label{item:trans_endl:FloXiXi}
			For all $i\in I$ we have \ $\sum\limits_{n=0}^\infty{\skal{T^n(t_{\xi_i})\xi_i}{\xi_i}}<\infty$.
		\item\label{item:trans_endl:XiXiGegenNull}
			For all $i\in I$ we have \ $\lim\limits_{n\to\infty}\skal{T^n(t_{\xi_i})\xi_i}{\xi_i}=0$.
		\item\label{item:trans_endl:FloOrginal}
			For all $i\in I$ we have \ $\lim\limits_{N\to\infty}\tfrac{1}{N}\sum\limits_{n=0}^{N-1}{\skal{T^n(t_{\xi_i})\xi_i}{\xi_i}}=0$.
	\end{enumerate}
\end{thm}

\begin{prf}
	\begin{description}[font=\normalfont\mdseries\em]
		\item[\ref*{item:trans_endl:trans}$\eqvt$\ref*{item:trans_endl:PotentialTraeger}:]
			Cf.\ \refrem{rem:trans}.
		\item[\ref*{item:trans_endl:Summierbar}$\folgt$\ref*{item:trans_endl:PotenzenGegenNull}:] Trivial.
		\item[\ref*{item:trans_endl:PotenzenGegenNull}$\folgt$\ref*{item:trans_endl:KeinInvarianterZustand}:]
			Let $\Fi\in\Zst(\cA)$ be a stationary state. Then $\Fi(p)=\Fi(T^n(p))\Geht{n\to\infty}0$.
		\item[\ref*{item:trans_endl:KeinInvarianterZustand}$\folgt$\ref*{item:trans_endl:PotentialTraeger}:]
			By assumption we have $p\leq(\supp\Fi)^\bot$ for all stationary states $\Fi\in\Zst(\cA)$. Hence $p\leq \ppr^\bot$ and applying \reflem{lem:pTr_summierbar_endl} we have $y:=\sum_{n=0}^\infty{T^n(p)}\in\cApot$. But, clearly, $p\leq\supp y$.
		\item[\ref*{item:trans_endl:PotentialTraeger}$\folgt$\ref*{item:trans_endl:Summierbar}:]
			Since $p\leq \ptr$, \reflem{lem:pTr_summierbar_endl} yields  $\sum_{n=0}^\infty{T^n(p)}\in\cA_+$.
		\item[\ref*{item:trans_endl:Summierbar}$\folgt$\ref*{item:trans_endl:t_XiSummierbar}$\folgt$\ref*{item:trans_endl:FloXiXi}$\folgt$\ref*{item:trans_endl:XiXiGegenNull}$\folgt$\ref*{item:trans_endl:FloOrginal}:] Trivial.
		\item[\ref*{item:trans_endl:FloOrginal}$\folgt$\ref*{item:trans_endl:KeinInvarianterZustand}:]
			Let $\Fi\in\Zst$ be a stationary state. By \refthm{thm:FloRecurrent} we have $\Fi(t_{\xi_i})=0$ for all $i\in I$. Since $p=\sum_{i\in I}{t_{\xi_i}}$ this yields $\Fi(p)=\sum_{i\in I}{\Fi(t_{\xi_i})}=0$.\qedhere
	\end{description}
\end{prf}

\section{Idempotent Markov Operators}\label{sec:Projections}

For a general unital completely positive projection $P$ on a C*-Algebra $\cA$, Choi and Effros have shown in \cite[Thm.\,3.1]{ce} that $P(\cA)$ can be turned into a C$^*$-algebra if a new multiplication $a\diamond b:= P(ab)$ is introduced on $P(\cA)$. In general, no further information about the structure of such maps seems to be available in the literature.

For a normal such map, i.e.\ an idempotent Markov operator $P$ on a von Neumann algebra $\cA\subseteq\BH$, we obtain a complete description of its structure. In particular, this will allow us to retrieve the Choi-Effros multiplication for such maps, thereby putting it into concrete terms.

Thus consider an idempotent Markov operator $P:\cA\to\cA$. Denote as above by $\ppr$ and $\ptr$ the maximal positive recurrent projection and the maximal transient projection for $P$, respectively, and set $\cApr:=\ppr\cA\ppr$ and $\cAtr:=\ptr\cA\ptr$.

If $\ppr=\Eins$ or, equivalently, if there is a faithful family of stationary normal states then it is well-known (cf.\ \cite[Thm.\,2.4]{kn}) that $P(\cA)$ is a von Neumann subalgebra and $P$ is a faithful normal conditional expectation from $\cA$ onto $P(\cA)$. If $\ppr\neq\Eins$ this does no longer need to be the case. As an example, let $\Fi\in\Zst(\cA)$ be some normal state on $\cA$ with support projection $p\neq\Eins$ and set $P(x):=pxp+\Fi(x)\,p^\bot$.

Remember that a conditional expectation $Q$ is called \emph{faithful} if $Q(x^*x)=0$ implies $x=0$.

\begin{thm}\label{thm:strukturVonProj}
	Let $P:\cA\to\cA$ be a Markov operator on a von Neumann algebra $\cA\subseteq\BH$ with $\ppr\neq\Eins$. Then the following conditions are equivalent:
	\begin{enumerate}[label=(\roman*)]
		\item\label{item:strukProj:idem} $P$ is idempotent, i.e.\ $P^2=P$.
		\item\label{item:strukProj:BedErw} There is a faithful normal conditional expectation $Q:\cApr\to\cApr$ and a completely positive unital normal map $S:Q(\cApr)\to\cAtr$ such that
		\[
			P(x)=Q(\ppr x\ppr)+S(Q(\ppr x\ppr)).
		\]
	\end{enumerate}
	\vspace{-2ex}
	In  this case $Q(x)=\ppr P(x)\ppr$ for $x\in\cApr$, hence $Q(\cApr)=\ppr P(\cA)\ppr$.
	
	This establishes a biunique correspondence between idempotent Markov operators on $\cA$ with maximal positive recurrent projection $\ppr\neq\Eins$ and pairs $(Q,S)$ where $Q$ is a faithful normal conditional expectation on $\cApr$ and $S:Q(\cApr)\to\cAtr$ a completely positive unital normal map. Moreover, $P$ is a conditional expectation if and only if $S$ is a *-homomorphism.
\end{thm}

For a normal state $\psi\in\Zst(\cA)$ with $\psi(\ppr)=1$ we denote its restriction $\psi\vert_{\cApr}\in\Zst(\cApr)$ by $\psi$ again. 

\begin{prf}
	\begin{description}[font=\normalfont\mdseries\em]
		\item[\ref*{item:strukProj:idem}$\folgt$\ref*{item:strukProj:BedErw}:]
			Since $\psi\circ P$ is a stationary normal state for every $\psi\in\Zst(\cA)$, we may write
			\[
				\ppr=\bigvee\Set{\,\supp(\psi\circ P)}{\psi\in\Zst(\cA)}.
			\]
			Hence $\psi(P(x))=\psi(P(\ppr x\ppr))$ for all $x\in\cA$ and $\psi\in\Zst(\cA)$. This implies $P(x)=P(\ppr x\ppr)$
			for every $x\in\cA$ and $P(x^*x)=0$ if and only if $\ppr x^*x\ppr=0$, since $(\psi\circ P)_{\psi\in\Zst(\cA)}$ is faithful on $\cApr$. Thus $\ppr$ is the \emph{support projection of $P$} as in \cite{es}.
			It follows that $P(\ptr)\leq P(\ppr^\bot)=0$, hence $P(\ppr)=\Eins$ and $\ppr^\bot$ is $P$-summable. Therefore, we have $\ppr+\ptr=\Eins$.

			The support projection of a normal unital positive projection on a JW-algebra has already been examined by Effros and St\o{}rmer: \cite[Lem.\,1.2.(2)]{es} implies that $\ppr$ commutes with all self adjoint elements in the range of $P$ and thus with all elements in $P(\cA)$. It follows that for all $x\in\cA$ we have
			\[
				P(x)=\ppr P(x)\ppr + \ptr P(x)\ptr.
			\]
			Let $Q:\,\cApr \to \cApr:\,x \mapsto \ppr P(x) \ppr$. Then $Q$ is normal, completely positive, and unital. Furthermore, $Q$ is idempotent, since $\ppr$ is the support projection of $P$ and hence for every $x\in\cApr$ we obtain
			\[
				Q^2(x) = \ppr P(\ppr P(x)\ppr)\ppr = \ppr P^2(x)\ppr = Q(x).
			\]
			Finally, if $\Fi\in\Zst(\cA)$ is stationary for $P$ then $\Fi$ is stationary for $Q$, too, since $\supp\Fi\leq\ppr$ and for every $x\in\cApr$ we have
			\[
				\Fi(Q(x))=\Fi(\ppr P(x)\ppr) = \Fi(P(x))=\Fi(x).
			\]
			Hence $Q$ has a faithful family of stationary normal states. Thus we infer from \cite[Thm.\,2.4]{kn} that $Q$ is a faithful normal conditional expectation onto the von Neumann subalgebra $Q(\cApr)=\ppr P(\cA)\ppr\subseteq\cApr$.
			
			Define $S:Q(\cApr)\ni x\mapsto \ptr P(x)\ptr\in\cAtr$. Then $S$ is normal,  completely positive, and unital. Summing up, for $x\in\cA$ we have:
			\begin{align*}
				P(x) & = P(\ppr x\ppr)=\ppr P(\ppr x\ppr)\ppr+\ptr P(\ppr x\ppr)\ptr \\
				& =\ppr P(\ppr x\ppr)\ppr+\ptr P(\ppr P(\ppr x\ppr)\ppr)\ptr \\
				& = Q(\ppr x\ppr)+\ptr P(Q(\ppr x\ppr))\ptr= Q(\ppr x\ppr)+ S(Q(\ppr x\ppr)).
			\end{align*}
		\item[\ref*{item:strukProj:BedErw}$\folgt$\ref*{item:strukProj:idem}:]
			We have to check that $P$ is idempotent. Since $\ppr S(Q(\ppr x\ppr))\ppr=0$, we obtain for $x\in\cA$:
			\begin{align*}
				P^2(x) & = P\bigl(Q(\ppr x\ppr)+S(Q(\ppr x\ppr))\bigr)\\
				&= Q(\ppr Q(\ppr x\ppr)\ppr)+S(Q(\ppr Q(\ppr x\ppr)\ppr))\\
				& = Q^2(\ppr x\ppr)+S(Q^2(\ppr x\ppr)) = P(x).\\
			\end{align*}
	\end{description}
	\vspace{-5ex}
	Clearly, this correspondence is biunique.
	
	It remains to prove the last assertion. An idempotent Markov operator $P$ is a conditional expectation if and only if $P(\cA)$ is multiplicatively closed. For each $x\in\cA$ set $x_Q:=Q(\ppr x\ppr)=\ppr P(\ppr x\ppr)\ppr=\ppr P(x)\ppr$. Then every $x\in P(\cA)$ is of the form $x=x_Q+S(x_Q)=\ppr x\ppr+S(\ppr x\ppr)$, since $x_Q=\ppr P(x)\ppr=\ppr x\ppr$. 

	If $S$ is a *-homomorphism then for $x,y\in\cA$ we have
	\begin{align*}
		P(x)P(y) &= (x_Q+S(x_Q))(y_Q+S(y_Q))=x_Q y_Q +S(x_Q)S(y_Q)\\
		& = x_Q y_Q + S(x_Q y_Q)\in P(\cA).
	\end{align*}
	On the other hand, if $S$ is not multiplicative then there are $x_Q,y_Q\in Q(\cApr)\subseteq\cA$ such that $S(x_Q)S(y_Q)\neq S(x_Q y_Q)$. Then $P(x_Q)=Q(\ppr x_Q\ppr)+S(Q(\ppr x_Q\ppr))=x_Q+S(x_Q)$ and
	\[
		\tilde x:=P(x_Q)P(y_Q)=x_Q y_Q +S(x_Q)S(y_Q)\neq x_Q y_Q +S(x_Q y_Q)=P(\tilde x),
	\]
	i.e.\ $P(x_Q)P(y_Q)=\tilde x\notin P(\cA)$. Hence $P(\cA)$ is not multiplicatively closed.
\end{prf}

\begin{cor}\emph{\textbf{(Choi-Effros product)}}\label{cor:ChoiEffrosProduct}
	Let $P:\cA\to\cA$ be an idempotent Markov operator on a von Neumann algebra $\cA\subseteq\BH$. Then $P(\cA)$ becomes an abstract von Neumann algebra with the new product $x\diamond y:=P(xy)$ for $x,y\in P(\cA)$ while the involution and Banach space structure are inherited from $\cA$.\\
	Moreover, the map $P\vert_{Q(\cApr)}$ is a  *-isomorphism from $\ppr P(\cA)\ppr=Q(\cApr)$ to $(P(\cA),\diamond)$.
\end{cor}

\begin{prf}
	As in the proof of \refthm{thm:strukturVonProj} set $x_Q:=Q(\ppr x\ppr)$ for $x\in\cA$. Then $x_Q$ and $S(x_Q)$ have orthogonal support projections and, as a vector space, $P(\cA)=\Set{x_Q+S(x_Q)}{x\in\cA}$ is isomorphic to the graph of $S$ and thus to $Q(\cApr)$. Employing our structure theorem for $P$, we obtain for $x,y\in P(\cA)$ more explicitly:
	\begin{align*}
		x\diamond y & = P\bigl((x_Q+S(x_Q))(y_Q+S(y_Q))\bigr) = P\bigl(x_Qy_Q+S(x_Q)S(y_Q)\bigr) \\
		& = P\bigl(x_Qy_Q\bigr)+P\bigl(\ppr S(x_Q)S(y_Q)\ppr\bigr) = P(x_Qy_Q) = x_Qy_Q+S(x_Qy_Q). 
	\end{align*}
	Hence the Choi-Effros product corresponds to the original one on $Q(\cApr)$. Therefore, it is associative. As in \cite{ce} we easily see that the norm is submultiplicative for this product and satisfies the C$^*$-property by using the corresponding properties for the original product and the Kadison-Schwarz inequality.
	Since $P$ is idempotent and normal, its range $(P(\cA),\diamond)$ is a weak*-closed *-algebra. This implies that $P(\cA)$ has a predual and, therefore, becomes an abstract von Neumann algebra by a well-known theorem of Sakai.
\end{prf}

The main work of Choi and Effros in \cite{ce} goes into showing that this product is associative. In our setting this follows more easily, since we have identified the Choi-Effros multiplication with the usual multiplication on $Q(\cApr)=\ppr P(\cA)\ppr$.
As pointed out to us by Izumi, another characterization of the Choi-Effros product was given by Arveson: He interpreted the Choi-Effros product as the usual multiplication on the fixed point algebra of a minimal dilation of $T$ (cf.\ \cite{izudial}).

For convenience we collect some consequences which follow immediately from the proof of \refthm{thm:strukturVonProj}.

\begin{cor}\label{cor:idempotentMarkovOp}
	Let $P^2=P:\cA\to\cA$ be a Markov operator on some von Neumann algebra $\cA\subseteq\BH$ with a decomposition as in  \refthm{thm:strukturVonProj}.
	\begin{enumerate}[label=(\alph*)]
		\item There exists no null recurrent part, i.e.\ $\ppr+\ptr=\Eins$.
		\item Any transient projection is mapped to zero. In particular, $P(\ptr)=0$.
		\item The maximal positive recurrent projection $\ppr$ is the \emph{support projection} $(\supp P)$ of $P$, i.e.\ $P(x)=P(\ppr x\ppr)=P(\ppr\,x\ppr\,\ppr)=P(x\ppr)=P(\ppr x)$ for all $x\in\cA$, and $P(x^*x)=0$ if and only if $\ppr x^*x\ppr=0$.
		\item The range of $P$ is a sum of the transient part and the positive recurrent part: $P(x)=\ppr P(x)\ppr+\ptr P(x)\ptr$ for every $x\in\cA$.
		\item For $x\in Q(\cApr)=\ppr P(\cA)\ppr$ we have $P(x)=x+S(x)\in P(\cA)$.
	\end{enumerate}
\end{cor}

As an illustration of such a decomposition, we include the following classical example (cf.\ also \refexmp{exmp:eb}).

\begin{exmp}\label{exmp:klassischeProjektion}
	Let $\Omega$ be a state space with three points and let $P:=\left(\begin{smallmatrix}1&0&\,0\\0&1&\,0\\ \!\nicefrac13&\nicefrac23&\,0\end{smallmatrix}\right)$ be the transition matrix of a classical Markov chain on $\Omega$. As in \refsec{sec:klassisch} the map $P=P^2$ can be regarded as a Markov operator on the algebra $\cA:=\CC^3$ of functions on $\Omega$. It is easily seen that $P(\cA) =\Span\set{\left(\begin{smallmatrix}1\\0\\\nicefrac13\end{smallmatrix}\right),\left(\begin{smallmatrix}0\\1\\\nicefrac23\end{smallmatrix}\right)}$ and $\ppr=\left(\begin{smallmatrix}1\\1\\0\end{smallmatrix}\right)$. Hence $\cApr=\CC^2\oplus0$, $\cAtr=0\oplus\CC$, and the decomposition according to \refthm{thm:strukturVonProj} is given by
	\[
		Q:\,\cApr\to\cApr:\;x\mapsto x\qquad\mbox{and}\qquad S:\,\cApr\to\cAtr:\;\left(\begin{smallmatrix}a\vphantom{\frac13}\\b\vphantom{\frac13}\\0\vphantom{\frac13}\end{smallmatrix}\right)\mapsto\left(\begin{smallmatrix}0\vphantom{\frac13}\\0\vphantom{\frac13}\\\frac13a+\frac23b\end{smallmatrix}\right).
	\]
\end{exmp}

\section{Non-Commutative Poisson Boundaries and Poisson Integrals}\label{sec:poisson}

It is natural to ask to what extent the results of the previous section can be carried over from the range of a projection to the fixed space $\fix T$ of a general Markov operator $T$. The main result in this section shows that a Choi-Effros product on $\fix T$ allows an identification as in the last statement of \refcor{cor:ChoiEffrosProduct} if and only if $T$ is weak* mean ergodic. In this case the identification may be viewed as an abstract Poisson integral.

If $T:\cA\to\cA$ is any Markov operator on a von Neumann algebra $\cA\subseteq\BH$ then there is a (not necessarily normal) projection $P$ in the pointwise weak*-closure of the convex hull $\co\Set{T^n}{n\in\NN}$ satisfying $TP=PT=P=P^2$ (see, e.g., \cite{es}). This implies $P(\cA)=\fix T$ and following \cite{luc} we call such $P$ an \emph{ergodic projection}. Even though $P$ may not be normal we define $\ppr(P)$ as in \refdefn{defn:recurrent}.

\begin{lem}\label{lem:pprT=pprP}
	Let $T:\cA\to\cA$ be a Markov operator on a von Neumann algebra $\cA\subseteq\BH$ and $P$ an ergodic projection for $T$. Then the maximal positive recurrent projections for $T$ and $P$ coincide, i.e.\ $\ppr(T)=\ppr(P)$.
\end{lem}

\begin{prf}
	Let $(\Phi_\alpha)$ be a net in $\co\Set{T^n}{n\in\NN}$ converging in the pointwise weak*-topo\-logy to an ergodic projection $P$ and let $\Fi\in\Zst(\cA)$. If $\Fi=\Fi\circ T$ then $\Fi=\Fi\circ \Phi_\alpha$, hence $\Fi=\Fi\circ P$.

	Conversely, if $\Fi=\Fi\circ P$ then $\Fi(T(x))=\Fi\circ P(T(x))=\Fi\circ P(x)=\Fi(x)$ for every $x\in\cA$. 
\end{prf}

If there exists a \emph{normal} ergodic projection for $T$ then this is the only ergodic projection and $T$ is called \emph{weak* mean ergodic} (cf.\ \cite{kn}). As in \refsec{sec:Projections} we abbreviate $\ppr\cA\ppr$ by $\cApr$.

\begin{prop}\label{prop:posRecFixRaumvNA}\emph{\cite{luc}}\ \
	Let $T:\cA\to\cA$ be a Markov operator on a von Neumann algebra $\cA\subseteq\BH$.
	\begin{enumerate}[label=(\alph*)]
		\item The map $\Tpr:\cApr\to\cApr:x\mapsto\ppr T(x)\ppr$ admits a faithful family of stationary normal states. In particular, $\Tpr$ is weak* mean ergodic.
		\item $\ppr\fix T\ppr=\fix\Tpr$ is a von Neumann algebra acting on the Hilbert space $\ppr\hrH$.
	\end{enumerate}
\end{prop}

This result is already contained in \cite[Thm.\,3,\:Cor.\,4]{luc} for arbitrary semigroups of positive (but not necessarily completely positive) contractions. For convenience we will give a proof which is adapted to our situation.

\begin{prf}
	The map $\Tpr$ is normal and completely positive and it is unital, since
	\[
		\Tpr(\ppr)=\ppr T(\ppr)\ppr = \ppr T(\ppr \Eins\ppr)\ppr \stackrel{\refKlammern{lem:charakSubharm}}{=} \ppr T(\Eins)\ppr=\ppr\Eins\ppr=\ppr. 
	\]
	Let $(\Fi_i)_{i\in I}\subseteq\Zst(\cA)$ be a family of stationary normal states for $T$ such that $\ppr=\bigvee_{i\in I}\supp\Fi_i$. Then $(\Fi_i)_{i\in I}$ is faithful on $\cApr$ and each $\Fi_i$ is stationary for $\Tpr$, because $\Fi_i(\Tpr(x))=\Fi_i(\ppr T(x)\ppr)=\Fi_i(T(x))=\Fi_i(x)$ for all $x\in\cApr$. This allows us to apply \cite[Thm.\,2.4]{kn} to obtain that $\Tpr$ is weak* mean ergodic and that $\fix{\Tpr}=\Set{x\in\cApr}{\Tpr(x)=x}$ is a von Neumann subalgebra of $\cApr$.

	From $\Tpr(\ppr x\ppr) 
	=\ppr T(x)\ppr=\ppr x\ppr$ for $x\in\fix T$ it follows that $\ppr\fix T\ppr\subseteq\fix\Tpr$.
	Conversely, let $y\in\fix{\Tpr}\subseteq\cA$. Then $y=\Tpr^n(y)\stackrel{\refKlammern{lem:charakSubharm}}{=}\ppr T^n(y)\ppr$ and thus $y=\ppr\Phi(y)\ppr$ for all $\Phi\in\co\Set{T^n}{n\in\NN}$. Let $P:\cA\to\cA$ be an ergodic projection for $T$ and $(\Phi_\alpha)$ a net in $\co\Set{T^n}{n\in\NN}$ converging to $P$ in the pointwise weak*-topology. Defining $x:=P(y)$ we have $T(x)=x$ and
	\[
		y = \wslim\nolimits_\alpha \ppr\Phi_\alpha(y)\ppr = \ppr\bigl( \wslim\nolimits_\alpha \Phi_\alpha(y) \bigr)\ppr = \ppr x\ppr \;\in\;\ppr\fix T\ppr.
	\]
	Hence $\ppr\fix T\ppr=\fix{\Tpr}$ is a von Neumann subalgebra of $\cApr$.
\end{prf}

The set of fixed points of a Markov operator $T$ can be turned into an abstract von Neumann algebra by using the Choi-Effros product w.r.t.\ an ergodic projection $P$ (cf.\ \refsec{sec:Projections}). This result was already proven by Effros and St{\o}rmer in the context of positive operators on JW-algebras \cite[Cor.\,1.6]{es}. More recently, this structure was studied and identified as \emph{non-commutative Poisson boundary} in \cite{izugrp} and \cite{izu} (cf.\ also \cite{arv04}).

\begin{thm}\label{thm:WeakStarMeanErgodic}
	Let $T:\cA\to\cA$ be a Markov operator on a von Neumann algebra $\cA\subseteq\BH$ and let $\ppr$ be its maximal positive recurrent projection. Then the following statements are equivalent:
	\begin{enumerate}[label=(\roman*)]
		\item\label{item:wme:wme} $T$ is weak* mean ergodic.
		\item\label{item:wme:ptr_to_eins} There is an ergodic projection $P$ such that $P(\ppr)=\Eins$.
		\item\label{item:wme:injective} The map $\fix T\ni x\mapsto\ppr x\ppr\in\ppr\fix T\ppr$ is injective (hence bijective).
		\item\label{item:wme:isom_effrosProd} There is an ergodic projection $P$ such that the map
		\[
			J_T:\,\ppr\fix T\ppr\to (\fix T,\diamond):\;x\mapsto P(x)
		\]
		is a *-isomorphism, where $a\diamond b:=P(ab)$ denotes the corresponding Choi-Effros product.
	\end{enumerate}
\end{thm}

In particular, the Poisson boundary $(\fix T,\diamond)$ of a weak* mean ergodic Markov operator $T$ can be faithfully represented on the  Hilbert space $\ppr\hrH$. More precisely, the von Neumann algebra $\ppr\fix T\ppr$ is a \emph{concrete realization} of $(\fix T,\diamond)$, as Izumi calls it (cf.\ \cite[Def.\,3.4]{izu}), if and only if $T$ is weak* mean ergodic.

\begin{prf}
	\begin{description}[font=\normalfont\mdseries\em]
		\item[\ref{item:wme:wme}$\folgt$\ref{item:wme:ptr_to_eins}:]
			Let $P$ be the (unique) normal ergodic projection for $T$. Then $\Fi\circ P$ is a stationary normal state (for $T$) for every $\Fi\in\Zst(\cA)$ and hence $\supp(\Fi\circ P)\leq\ppr$. Therefore, we have $\Fi(P(\ppr))=\Fi(P(\Eins))$ for all $\Fi\in\Zst(\cA)$, which implies $P(\ppr)=P(\Eins)=\Eins$.
		\item[\ref{item:wme:ptr_to_eins}$\folgt$\ref{item:wme:injective}:]
			By assumption we have $P(\ppr^\bot)=0$. Let $a\in\cA$ then the Kadison-Schwarz inequality implies
			\[
				0\leq P(a\,\ppr^\bot)^*P(a\,\ppr^\bot)\leq P(\ppr^\bot a^*a\,\ppr^\bot)\leq \norm{a}^2\,P(\ppr^\bot)=0.
			\]
			Hence $P(a\,\ppr^\bot)=0$ and, analogously, $P(\ppr^\bot a)=0$. It follows that for $x\in\fix T$
			\[
				x=P(x)=P\bigl(\ppr x\ppr + (\ppr x)\ppr^\bot + \ppr^\bot x\bigr)=P(\ppr x\ppr)
			\]
			and, therefore, the linear map $\fix T\ni x\mapsto\ppr x\ppr\in\ppr\fix T\ppr$ is injective.
		\item[\ref{item:wme:injective}$\folgt$\ref{item:wme:wme}:]
			We show that the stationary normal linear functionals separate the points of $\fix T$ which is equivalent to Condition \ref{item:wme:wme} by \cite[Thm.\,1.2]{kn}. Let $0\neq x\in\fix T$ then $0\neq\ppr x\ppr\in\fix{\Tpr}$ by assumption. Since $\Tpr$ is weak* mean ergodic, there is a normal linear functional $\tilde\Fi\in(\cApr)_*$ with $\tilde\Fi\circ\Tpr=\tilde\Fi$ and $\tilde\Fi(\ppr x\ppr)\neq0$.
			Let $\Fi:=\tilde\Fi(\ppr\,\cdot\,\ppr)\in\cA_*$ then
			\[
				\Fi(T(a))=\tilde\Fi(\ppr T(a)\ppr)\stackrel{\refKlammern{lem:charakSubharm}}{=}\tilde\Fi(\ppr T(\ppr a\,\ppr)\ppr)
				=\tilde\Fi(\ppr a\,\ppr)=\Fi(a)
			\]
			for all $a\in\cA$. Hence $\Fi$ is a stationary normal linear functional for $T$ and $\Fi(x)=\tilde\Fi(\ppr x\ppr)\neq0$.
		\item[\ref{item:wme:wme}$\folgt$\ref{item:wme:isom_effrosProd}:]
			If $T$ is weak* mean ergodic then there is a unique normal ergodic projection $P$. By \refcor{cor:ChoiEffrosProduct} it follows that $J_T:x\mapsto P(x)$ is a *-isomorphism from $\ppr\fix T\ppr$ to $(\fix T,\diamond)$.
		\item[\ref{item:wme:isom_effrosProd}$\folgt$\ref{item:wme:ptr_to_eins}:]
			Since $J_T$ is a *-isomorphism, it follows that $\Eins=J_T(\ppr)=P(\ppr)$.\qedhere
	\end{description}
\end{prf}

{\L}uczak has already proven the equivalence of \ref{item:wme:wme} and \ref{item:wme:ptr_to_eins} in \cite[Thm.\,5]{luc}, where he gave a different proof. The equivalence of \ref{item:wme:wme} and \ref{item:wme:injective} is implicitly contained in \cite{fv}.

As in the classical situation the operator $J_T$ defined in Condition \ref{item:wme:isom_effrosProd} extends elements of the Poisson boundary uniquely to an element of the fixed space of the Markov operator. Therefore, for a weak* mean ergodic Markov operator $T$ we may view
\[
	J_T:\, \ppr\fix T\ppr\to\fix T:\; x\mapsto P(x)=x+S(x)
\]
as a \emph{non-commutative} or \emph{abstract Poisson Integral}, where $S$ comes from the decomposition of the normal ergodic projection $P$ (cf.\ \refthm{thm:strukturVonProj}).

The following example shows that it is not enough to require the existence of an arbitrary *-iso\-mor\-phism from $\ppr\fix T\ppr$ to $(\fix T,\diamond)$ to ensure that a Markov operator $T$ is weak* mean ergodic, i.e.\ \refXX{Condition}{thm:WeakStarMeanErgodic}{.\ref*{item:wme:isom_effrosProd}} cannot be weakened in this way.

\begin{exmp}\label{exmp:IsomPoissonBoundaryReichtNicht}
	Let $\cA:=\ell^\infty(\ZZ)=\ell^\infty(-\NN)\oplus\ell^\infty(\NN_0)$ and $T_0$ the left shift on $\ell^\infty(\NN_0)$, i.e. $T_0(f)(n)=f(n+1)$ for $f\in\ell^\infty(\NN_0)$. For $x=x^-\oplus x^+\in\cA=\ell^\infty(-\NN)\oplus\ell^\infty(\NN_0)$ we define a Markov operator $T:\cA\to\cA$ by
	\[
		T(x)=T(x^-\oplus x^+)=x^-\oplus T_0(x^+).
	\]
	Then $\ppr=\ChFkt_{(-\NN)}$, since there are no normal stationary states for the shift $T_0$ (cf.\ \refrem{rem:shift_auf_ell_infty}). Furthermore, we have $\fix T=\ell^\infty(-\NN)\oplus\CC\cdot\Eins_{\ell^\infty(\NN_0)}$. Since $\fix T$ already is a subalgebra, the Choi-Effros product coincides with the usual one.
	
	Clearly, $\fix T$ is *-isomorphic to $\ell^\infty(-\NN)=\ppr\fix T\ppr$.
	But $T$ is not weak* mean ergodic, since $\varphi(0\oplus\Eins_{\ell^\infty(\NN_0)})=0$ for all stationary normal states $\varphi=\varphi\circ T$.
\end{exmp}

\begin{prop}\label{prop:ptr_potential_falls_T_wme}
	If $T:\cA\to\cA$ is a weak$^*$ mean ergodic Markov operator on a von Neumann algebra $\cA\subseteq\BH$ then $\pnr=0$, i.e. there is no null recurrent part.\\
	Furthermore, the maximal transient projection for $T$ is a potential, i.e.\ $\ptr\in\cApot$.
\end{prop}

\begin{prf}
	Since $\ppr^\bot$ is superharmonic (cf.\ \refthm{thm:superharm_complete_lattice}), there are elements $y\in\cApot$ and $0\leq h\in\fix T$ such that $\ppr^\bot=y+h$ by the Riesz decomposition theorem \refKlammern{thm:RieszDecompThm}.

	Let $P:\cA\to\cA$ be the normal ergodic projection for $T$. Then $P(\ppr)=\Eins$ by \refthmX{thm:WeakStarMeanErgodic}{.\ref*{item:wme:ptr_to_eins}} and hence
	\[
		0=P(\ppr^\bot)=P(y+h)=P(y)+h\geq 0.
	\]
	This yields $h=0$ and thus $\ppr^\bot=y\in\cApot$. Therefore, $\ppr^\bot$ is transient, hence $\pnr=0$ and $\ptr=\ppr^\bot$.
\end{prf}

Using the Riesz decomposition theorem \refKlammern{thm:RieszDecompThm} and \refcor{cor:potential_dominierter_fixpkt_verschwindet} we conclude that the support projection of every potential for a weak$^*$ mean ergodic Markov operator is a potential, too.

\section{Entanglement Breaking Channels}\label{sec:eb}

As another application of our theory we consider entanglement breaking channels (see \cite{hol}, \cite{hsr}, \cite{rus}, or \cite{arv}).

\begin{defn}
	Let $\cA\subseteq\BH$ be a von Neumann algebra and $T:\cA\to\cA$ a Markov operator. If there are $\Fi_1,\ldots,\Fi_n\in\Zst(\cA)$ and $a_1,\ldots,a_n\in\cA_+$ with $\sum_{i=1}^n{a_i}=\Eins$ such that for all $x\in\cA$
	\begin{align*}
		T(x)=\sum\nolimits_{i=1}^n{\Fi_i(x)a_i}  \TAG{eqn:HolevoRep}
	\end{align*}
	then we call \refKlammern{eqn:HolevoRep} a \emph{Holevo representation} for $T$. 
\end{defn}

Horodecki, Shor, and Ruskai showed in \cite{hsr} that if $\hrH$ is finite dimensional and $T:\BH\to\BH$ a Markov operator then $T$ has a Holevo representation if and only if $\rho\circ (T\otimes\id)$ is a \emph{separable} state (i.e.\ a convex combination of product states) for every state $\rho$ on $\hrB(\hrH\otimes\hrH)$. For this reason such maps are also called entanglement breaking.

Note that being of finite rank, a Markov operator $T:\cA\to\cA$ with Holevo representation is automatically weak* mean ergodic.

\begin{thm}\label{thm:EB:projImFixraumKomm}
	Let $T:\cA\to\cA$ be a Markov operator with Holevo representation. Then any two orthogonal projections in the fixed space $\fix T$ commute.\\
	Moreover, $\ppr\fix T\ppr$ is a finite dimensional commutative von Neumann subalgebra of $\cApr=\ppr\cA\ppr$.
\end{thm}

Arbitrary elements of the fixed space $\fix T$ do not need to commute (see \refexmp{exmp:eb} below) even though $\fix T$ is a commutative von Neumann algebra w.r.t.\ the Choi-Effros product, being isomorphic to $\ppr\fix T\ppr$ (cf.\ \refthm{thm:WeakStarMeanErgodic}).

\begin{prf}
	As first step we show that for $p^*=p^2=p\in\fix T$ there is a subset $J_p\subseteq\set{1,\ldots,n}$ such that $p=\sum\nolimits_{i\in J_p}{a_i}$ and $pa_i=a_ip$ for all $1\leq i\leq n$. The first assertion is an immediate consequence of this fact.

	Define $J_p:=\Set{1\leq i\leq n}{\Fi_i(p)\neq 0}$ then $p=T(p)=\sum_{i\in J_p}{\Fi_i(p)a_i}$. Hence $a_i\leq \supp a_i\leq p$ and thus $pa_i=a_i=a_ip$ for $i\in J_p$. Together with $\sum_{i=1}^n a_i=\Eins$ this implies $\sum_{i\in J_p}{a_i}\leq p = \sum_{i\in J_p}{\Fi_i(p)a_i}$ from which $\Fi_i(p)=1$ for $i\in J_p$ follows, i.e. $p= \sum_{i\in J_p}{a_i}$. 
	Setting $J_{p^\bot}:=\Set{1\leq i\leq n}{\Fi_i(p)= 0}$ it follows that $p^\bot= \sum_{i\in J_{p^\bot}}{a_i}$. Hence $p$ also commutes with all $a_i$, $i\in J_{p^\bot}$.

	Let $\Tpr:\cApr\to\cApr:x\mapsto\ppr T(x)\ppr$ 
	as in \refprop{prop:posRecFixRaumvNA}. Then $\Tpr(x)=\ppr T(x)\ppr = \sum\nolimits_{i=1}^n{\Fi_i(x)\ppr a_i\ppr}$ for $x\in\cApr$. Hence $\Tpr$ has a Holevo representation, too, and any two orthogonal projections of $\fix{\Tpr}$ commute. Since $\fix{\Tpr}=\ppr\fix T\ppr$ is a von Neumann subalgebra of $\cApr$ and since a von Neumann algebra is generated by its projections, this completes the proof.
\end{prf}

\begin{exmp}\label{exmp:eb}
	Let $\cA:=M_5(\CC)$ and let  $(e_{ij})_{1\leq i,j\leq5}$ be the canonical system of matrix units. Set $\Fi_1:=\tr(e_{11}\,\cdot\,)$, $\Fi_2:=\tr(e_{22}\,\cdot\,)$, and $\Fi_3:=\tr(e_{33}\,\cdot\,)$, as well as
	\[
		a_1:=\frac14\mbox{\begin{footnotesize}$\left(\begin{matrix} 4 & 0 & 0 & 0 & 0\\0 & 0 & 0 & 0 & 0\\0 & 0 & 0 & 0 & 0 \\0 & 0 & 0 & 2 & 1\\0 & 0 & 0 & 1 & 1\end{matrix}\right)$,\end{footnotesize}}\quad
		a_2:=\frac14\mbox{\begin{footnotesize}$\left(\begin{matrix} 0 & 0 & 0 & 0 & 0\\0 & 4 & 0 & 0 & 0\\0 & 0 & 0 & 0 & 0\\0 & 0 & 0 & 2 & -1\\0 & 0 & 0 & -1 & 1\end{matrix}\right)$,\end{footnotesize}}\quad
		a_3:=\frac14\mbox{\begin{footnotesize}$\left(\begin{matrix} 0 & 0 & 0 & 0 & 0\\0 & 0 & 0 & 0 & 0\\0 & 0 & 4 & 0 & 0\\0 & 0 & 0 & 0 & 0\\0 & 0 & 0 & 0 & 2 \end{matrix}\right)$\end{footnotesize}}.
	\]
	Then $P:M_5(\CC)\to M_5(\CC):\,x\mapsto \Fi_1(x)a_1+\Fi_2(x)a_2+\Fi_3(x)a_3$ is an idempotent Markov operator in Holevo representation. Note that $\fix P=P(\cA)=\Span\set{a_1,a_2,a_3}$ contains no orthogonal projections except for $0$ and $\Eins$. Obviously, the states $\Fi_1$, $\Fi_2$, and $\Fi_3$ are stationary and $\ppr=\bigvee_{i=1}^3\supp\Fi_i=e_{11}+e_{22}+e_{33}$. Hence as an algebra $\ppr\fix P\ppr\cong\CC^3$ but neither $a_1$ and $a_2$, nor $a_1$ and $a_3$, nor $a_2$ and $a_3$ do commute. However, the Choi-Effros product turns the elements $a_1$, $a_2$, and $a_3$ into orthogonal projections which are mutually orthogonal.
\end{exmp}

For a Markov operator which is idempotent we have a converse of \refthm{thm:EB:projImFixraumKomm}.

\begin{thm}\label{thm:EB:PmitKommutativerUnteralg}
	Let $\cA\subseteq\BH$ be a von Neumann algebra and $P^2=P:\cA\to\cA$ a Markov operator such that $\ppr P(\cA)\ppr$ is finite dimensional and commutative.\\
	Then $P$ has a Holevo representation and $a_1,\ldots,a_n\in\cA_+$ and $\Fi_1,\ldots,\Fi_n\in\Zst(\cA)$ can be chosen linearly independent.
\end{thm}

\begin{prf}
	Let $P$ have a decomposition as in \refthm{thm:strukturVonProj} and let $p_1,\ldots,p_n$ be a family of minimal mutually orthogonal projections generating $\ppr P(\cA)\ppr=Q(\cApr)$. Then for each $1\leq i\leq n$ there is is normal state $\Fi_i\in\Zst(\cA)$ such that $p_i Q(\ppr x\ppr)p_i=p_i P(x)p_i= \Fi_i(x)p_i$ for all $x\in\cA$.
	
	Moreover, $\Fi_i(x)p_i=p_i Q(\ppr x\ppr)p_i=p_i Q(p_i x p_i)p_i=\Fi_i(p_i x p_i)p_i$ for all $x\in\cA$, because $Q$ is a conditional expectation. Hence $\supp\Fi_i\leq p_i\leq\ppr$, in particular, $\Fi_1,\ldots,\Fi_n$ are linearly independent. Furthermore, for $x\in\cApr$ we have $Q(x)=\ppr Q(x)=\sum_{i=1}^n{p_i Q(x)}=\sum_{i=1}^n{p_i Q(x)p_i}=\sum_{i=1}^n{\Fi_i(x)p_i}$.
	
	Let $b_i:=S(Q(p_i))=S(p_i)$, where $S$ comes from the decomposition of $P$, and set $a_i:=p_i+b_i=P(p_i)$. Obviously, $a_1,\ldots,a_n$ are linearly independent and
	\begin{align*}
		P(x) & = Q(\ppr x\ppr)+ S(Q(\ppr x\ppr))= \sum_{i=1}^n{\Fi_i(\ppr x\ppr)p_i}+ S\biggl(\sum_{i=1}^n{\Fi_i(\ppr x\ppr)p_i}\biggr) \\
		& = \sum_{i=1}^n{\Fi_i(x)p_i}+ \sum_{i=1}^n{\Fi_i(x)S(p_i)}= \sum_{i=1}^n{\Fi_i(x)(p_i+b_i)}= \sum_{i=1}^n{\Fi_i(x)a_i}.\qedhere
	\end{align*}
\end{prf}

If $P$ is already a conditional expectation such that $P(\cA)$ is commutative then $P\otimes\id$ projects onto $P(\cA)\otimes M_n$. Since every state on this algebra is separable, such a $P$ is entanglement breaking and thus has a Holevo representation. In general, however, $P(\cA)$ is not necessarily contained in a commutative subalgebra as \refexmp{exmp:eb} shows.

The naturally arising question is whether such a representation is unique. The following proposition is answering this question affirmatively.

\begin{prop}
	Let $P^2=P:\cA\to\cA$ be a Markov operator which admits a Holevo representation $P(x)=\sum\nolimits_{i=1}^n{\Fi_i(x)a_i}$ with  linearly independent normal states $\Fi_i$ and linearly independent positive elements $a_i$ $(1\leq i\leq n)$. Then $a_1,\ldots,a_n\in\cA_+$ and $\Fi_1,\ldots,\Fi_n\in\Zst(\cA)$ are uniquely determined up to permutation.
\end{prop}

\begin{prf}
For all $x\in\cA$ we have
\begin{align*}
	\sum_{i=1}^n{\Fi_i(x)a_i} & = P(x) = P^2(x) = \sum_{i=1}^n{\Fi_i(P(x))a_i} \\
	& = \sum_{i=1}^n{\Fi_i\biggl(\sum_{j=1}^n{\Fi_j(x)a_j}\biggr)a_i} = \sum_{i=1}^n{\biggl(\sum_{j=1}^n{\Fi_j(x)\Fi_i(a_j)}\biggr)a_i}.
\end{align*}

Since $a_1,\ldots,a_n$ are linearly independent, for every $i\leq n$ and $x\in\cA$ this yields  $\Fi_i(x)=\sum_{j=1}^n{\Fi_j(x)\Fi_i(a_j)}$, or, equivalently,
\[
	\left(\Fi_i(a_i)-1\right)\Fi_i(x)+\sum_{j\neq i}{\Fi_i(a_j)\Fi_j(x)} = 0.
\]
Now the linear independence of $\Fi_1,\ldots,\Fi_n$ implies $\Fi_i(a_j)=\delta_{ij}$ (Kronecker delta). It follows easily that $P(a_i)=a_i$ and $\Fi_i\circ P=\Fi_i$ for every $1\leq i\leq n$.
On the other hand, every stationary state is a linear combination of $\Fi_1,\ldots,\Fi_n$. Therefore, we have $\ppr=\bigvee_{i=1}^n\supp\Fi_i$.

Let $p_i:=\supp\Fi_i$, then $p_i\leq a_i$ and thus $\Fi_j(p_i)\leq\Fi_j(a_i)=0$ for $i\neq j$. Therefore, we have $p_i\leq a_i\leq p_j^\bot$ for all $1\leq i,j\leq n$ with $i\neq j$, i.e.\ the support projections of $\Fi_1,\ldots,\Fi_n$ are mutually orthogonal and $\ppr=\sum_{j=1}^n p_j$.

Clearly, $p_1,\ldots,p_n$ commute and linearly span $\ppr P(\cA)\ppr$. Hence we have shown that $p_i=\ppr a_i\ppr$ and thus $a_i=P(a_i)=Q(\ppr a_i\ppr)+S(Q(\ppr a_i\ppr))=p_i+S(p_i)$ for all $i\leq n$. So $a_i$ has the same form as in the proof of \refthm{thm:EB:PmitKommutativerUnteralg} which completes the proof.\qedhere
\end{prf}

\section*{Acknowledgment}\vspace*{-2ex}
We gratefully acknowledge a discussion with Masaki Izumi, who kindly sent us a preprint version of his review on $E_0$-semigroups dedicated to the memory of W. Arveson \cite{izudial}, and pointed to us Arveson's interpretation of the Choi-Effros product (cf.\ \refcor{cor:ChoiEffrosProduct}). We are also indebted to Florian Haag, who agreed that we included \refthm{thm:FloRecurrent}, a result he obtained in his Master studies supervised by the second author. Finally, our thanks go to Hans Maassen for his careful reading of a version of this paper.

\label{LastPage}

\end{document}